\documentclass[12pt]{article}
\usepackage{hyperref}
\usepackage[utf8]{inputenc}
\usepackage{amsmath,amssymb}
\usepackage[active]{srcltx}
\usepackage{times}
\usepackage{array}
\usepackage[english]{babel}
\usepackage{wrapfig} 
\usepackage{amssymb,amsthm,color}
\input epsf
\usepackage{graphicx}
\usepackage{pdfsync}
\usepackage{changes}
\definecolor{wineRed}{rgb}{0.7,0,0.3}

\theoremstyle{definition}
\newtheorem{theorem}{Theorem}[section]
\newtheorem{defn}[theorem]{Definition}
\newtheorem{lemma}[theorem]{Lemma}
\newtheorem{proposition}[theorem]{Proposition}
\newtheorem{cor}[theorem]{Corollary}
\newtheorem{exe}[theorem]{Example}

\newtheorem{remark}[theorem]{Remark}

\def\conv{{\hbox{conv}\,}}

\def\diam{\hbox{diam}\,}
\def\dist{\hbox{dist}\,}

\def\bN{\mathbb{N}}
\def\bR{\mathbb{R}}

\def\d{\partial}\def\p{\partial}

\def\cH{{\mathcal{H}}}
\def\cI{{\mathcal{I}}}
\def\cK{{\mathcal{K}}}

\def\cN{{\mathcal{N}}}

\numberwithin{equation}{section}
\newcommand{\subjclass}[1]{\bigskip\noindent\emph{2010 Mathematics Subject Classification:}\enspace#1}
\newcommand{\keywords}[1]{\noindent\emph{Keywords:}\enspace#1}

\begin{document}
\title{The planar Least Gradient problem in convex domains, the case of continuous datum}
\author{Piotr Rybka\footnote{Faculty of Mathematics, Informatics and  Mechanics, the University of Warsaw,
ul. Banacha 2, 02-097 Warsaw, Poland, email: rybka@mimuw.edu.pl}, Ahmad Sabra\footnote{Department of mathematics, American University of Beirut, Riad el solh, 1107 2020 Beirut, Lebanon. email: asabra@aub.edu.lb}
\date{}}

\maketitle
\begin{abstract}
We study the two dimensional least gradient problem in a convex polygonal set in the plane. We show  existence of solutions when the boundary data are attained in the trace sense. Due to the lack of strict convexity, the classical results are not applicable. 
We state the admissibility conditions on the continuous boundary datum  $f$ that are sufficient for establishing an existence and uniqueness result. The solutions are constructed  by a limiting process, which uses  the well-known geometry of superlevel sets of least gradient functions.

 \subjclass{Primary: 49J10, Secondary: 49J52, 49Q10, 49Q20}

 \keywords{least gradient, trace solutions, convex  but not strictly convex domains, $BV$ functions.}
 
\end{abstract}

\section{Introduction}
We study  the least gradient problem
\begin{equation}\label{lg}
 \min \left\{ \int_\Omega |D u|: \ u\in BV(\Omega),\ T u = f\right\},
\end{equation}
where $\Omega$ is a bounded convex region in the plane  with polygonal boundary. 
We denote by $T: BV(\Omega) \to
L^1(\partial\Omega)$ the trace operator. 
We stress that we are interested  in solutions to (\ref{lg}) satisfying the boundary conditions in the sense of trace of $BV$ functions, 
i.e.,
$ T u = f,$  where $f$ is in $C 
(\partial\Omega)$. 

A motivation to study \eqref{lg} comes from the conductivity problem and free material design, see \cite{grs} and the references therein. A weighted least gradient problem appears in medical imaging, \cite{nachman}, and \cite{timonov}, which requires investigating the anisotropic version of \eqref{lg}, see \cite{jerrard}.

Since the publishing of the paper by Sternberg-Williams-Ziemer, \cite{sternberg}, the least gradient problem was broadly studied. In \cite{sternberg},  existence and uniqueness for continuous data $f$ were shown when the boundary of the region $\Omega$ has a non-negative mean curvature (in a weak sense) and $\d\Omega$ is not locally area minimizing. The anisotropic case was studied in \cite{jerrard}, and  \cite{moradifam}, where $\Omega$ satisfies a barrier condition that is equivalent to the conditions in \cite{sternberg} for the isotropic setting. However, these approaches are not applicable here since the barrier condition in $\bR^2$ reduces to strict convexity of $\Omega$ which is violated in the domains of interest in this paper.

The authors of \cite{spradlin} showed that the space of traces of solutions to the least gradient problem is essentially smaller than
$L^1(\partial\Omega)$ when $\Omega$ is a disk. As a result, solutions to \eqref{lg} do not necessarily exist for all $L^1$-data. Using a weaker interpretation of the boundary conditions, the authors of \cite{mazon} proved the existence of solutions to a relaxed least gradient problem for general Lipschitz domain with $L^1$ boundary data, see \cite[Definition 2.3]{mazon}. Moreover, the example in \cite{mazon} shows that even a finite number of discontinuity points leads to the loss of  uniqueness of solutions.  However, a recent article \cite{gorny2} provides a classification of multiple solutions. This result is valid for convex regions, which  need not be strictly convex.

The regularity of solutions to \eqref{lg} was also studied. In \cite{sternberg}, when $\Omega\subseteq\bR^d$,  $d\ge2$, and the boundary  has strictly positive mean curvature, then the authors showed that if $f\in C^{\alpha}(\partial \Omega)$, then the solution $u\in C^{\alpha/2}(\Omega)$. In \cite{grs}, a connection between the least gradient problem and the Beckmann minimal flow problem was established in $\bR^2$. An implication of this result is the formulation of \eqref{lg} in a mass transport setting, this was done in \cite{Filippo}. Using this fact, the authors in \cite{Filippo} proved existence and uniqueness of solutions to \eqref{lg} when $f\in W^{1,p}(\partial \Omega)$ and $\Omega$ is uniformly convex, $p\leq 2$. In this case the solution $u$ belongs to $W^{1,p}(\Omega)$.

A common geometric restriction on the domain $\Omega$ in the mentioned literature boils down to strict convexity in planar domains. 
In the present paper, the main difficulty is the lack of strict convexity of a convex region $\Omega$. 
In general, we do not expect 
existence of solutions to \eqref{lg}, even in the case of continuous $f$, once the strict convexity condition is dropped, see \cite[Theorem 3.8]{sternberg}. As a result,  we
have to develop a proper tool to examine the domain and the range of the data. For this purpose, we state admissibility
conditions on the behavior of $f$ on the sides of the boundary. In order to avoid unnecessary technical difficulties, we restrict our attention to bounded polygons that have finite or infinite number of sides. 
In the present paper we analyze (\ref{lg}) when the  data are continuous. In this case  our assumptions are as follows: $f$ satisfies the
admissibility condition (C1), see Definition \ref{admC1},  which implies  monotonicity of $f$ on the sides
of $\d\Omega$. However, if monotonicity is violated on a side of $\d\Omega$, a solution can still be constructed provided that $f$ satisfies  admissibility condition (C2), see Definition \ref{admC2}. Condition (C2) is geometric and it is related to the range of $f$ in $\Omega$. Mainly, it means that the data on a flat boundary may attain maxima or minima on large sets, making creation of
level sets of a positive Lebesgue measure advantageous,  see Section \ref{sCdata}. 
It turns out that condition (C2) has to be complemented, 

for the purpose of  killing a category of bad data. For this reason, we  introduce the Ordering Preservation Condition (OPC for short), see Definition \ref{dfopc}, and the Data Consistency Condition (DCC for short), see   Definition  \ref{dfdcc}.

Roughly speaking,
the OPC condition does not permit datum $f$, which leads to intersections of the level sets of the candidates of solutions. At the same time, the DCC 

says that, if $f$ attains a maximum on side $\ell_1$, then $f$ may  not  attain any minimum on side $\ell_2$ in front of  $\ell_1$. In the last section, we present  examples of solutions, as well as we show that dropping (C1), (C2), OPC or DCC leads to the non-existence of solutions.

The paper is organized as follows. In Section \ref{adm}, 
we introduce the admissibility conditions for continuous functions. We first investigate the existence and uniqueness of solutions, when the set $\Omega$ has finitely many sides, see Section \ref{finite Cont}. Subsequently, we deal with the boundary of $\d\Omega$ having an  infinite number of sides,  see Section \ref{infinite Cont}. In this case,  we make an additional assumption. Namely, we assume that the sides accumulate at one point and the boundary datum $f$ has a strict extremum at the accumulation point.
In the two cases mentioned above, we are also assuming that $f$ has finitely many humps. By a {\it hump} we mean a closed interval on which $f$ attains a local maximum or minimum, see Definition \ref{dehu} for details. In section \ref{inf hump}, we study the case of oscillatory data where the function $f$ has infinitely many humps on one side of $\d\Omega$. Here, we need to control the oscillations of $f$. We do so by  requiring that $f$ belongs to $BV(\d\Omega)$.

We close the paper by presenting in Section \ref{examples} a number of simple examples of boundary data, where our function $f$ violates one of the admissibility conditions, and we show that in these cases a solution might fail to exist. 

The results of this paper are summarized in the following theorem:

\begin{theorem}\label{tsuf-C} 
Let us suppose that $f\in C(\d\Omega)$, 
$\Omega$ is an open, bounded and convex set.  The boundary of $\Omega$ is polygonal, and  may  consist of a finite or infinite number of sides,  $\d \Omega = \bigcup_{j\in \mathcal{J}} \ell_j$, where $\ell_j$ are line segments. Furthermore,  $f$ satisfies the
admissibility conditions (C1) or  (C2) on all sides of $\d\Omega$,  the Ordering Preservation Condition (\ref{OPC}) and the Data Consistency Condition (\ref{dcc1}-\ref{dcc2}).\\
(a) If the number of sides  as well as the number of humps are finite, then there exists a unique solution to  problem (\ref{lg}).\\
(b) We assume that the number of sides is infinite,  there is
at most one accumulation point $p_0$. We also require  that the number of all humps of $f$ is finite
and $f$ has 
a local maximum/minimum at $p_0$. Moreover, there is $\epsilon>0$ such that the restriction of $f$ to each component of $(B(p_0,\epsilon)\cap\d\Omega)\setminus \{p_0\}$ is strictly monotone.
Under these assumptions the problem (\ref{lg}) has a unique solution.\\
(c) Suppose that only one side, $\ell$, has an infinite number of humps accumulating at its endpoint $p_0$. Point $p_0$ may be an accumulation point of sides of $\Omega$. 
If 
in addition $f\in BV(\d\Omega)$,  then problem (\ref{lg}) has a unique solution.
\end{theorem}

Parts  (b) and (c) say that we can deal with an infinite number of sides, but we need additional assumptions for this purpose. 
The strategy of our proof is as follows. In part (a), we construct a sequence of strictly convex regions $\Omega_n$ converging to $\Omega$ in the Hausdorff distance. We also provide approximating data on $\d\Omega_n$. By the classical result in \cite{sternberg}, we obtain a sequence of continuous solutions $v_n$ to the least gradient problem (\ref{lg}) on $\Omega_n$.  After estimating the common modulus of continuity, we may pass to the limit using a result by \cite{miranda}. 

We cannot apply the same approach in part (b), because  the estimate on the continuity modulus of the solutions to the least gradient problem depends on the number of sides. Thus, the approximation technique we use has its limitations. This is why we restricted our attention to the cases listed in Theorem \ref{tsuf-C}.

In part (b) and (c), we approximate $\Omega$ by an increasing sequence $\{\Omega_n\}_{n=1}^\infty$ of regions  satisfying   the conditions of (a). As a result we obtain a sequence of corresponding unique solutions, $u_n$, to the Least Gradient Problem. Moreover, functions $u_n$ are  such that $u_{n+1}|_{\Omega_n} = u_n$. This condition guarantees convergence of sequence $\{u_n\}_{n=1}^\infty$ to a limit $u$ which  has the correct trace.

It is a natural and  interesting question to ask if the sufficient conditions, we specified in the above theorems, are also necessary and if we can generalize our result to discontinuous data.  
However, these problems require a quite different technique and are beyond the scope of this paper. These questions will be addressed elsewhere. 

\section{Admissibility criteria}\label{adm}

We assume that the region $\Omega\subset \bR^2$ is convex and its boundary $\d\Omega$ is a polygonal curve, i.e, it  is a union of line segments,
$$
\d \Omega = \bigcup_{j\in \mathcal{J}} \ell_j.
$$
The number of sides may be finite or infinite and  the  line segments $\ell_j$, $j\in \mathcal{J}$ are closed.

Here, we will deal only with continuous data $f\in C(\d\Omega)$ for the problem (\ref{lg}). We stress that we are interested only in solutions, such that the boundary condition is assumed in the trace sense. For this reason, it is important to monitor the behavior of $f$ on sides of  the polygon $\d\Omega$. 
We expect that data must satisfy some sort of admissibility 
conditions on $\partial \Omega$. 
We will state them in this section.

\begin{defn} \label{admC1} We shall say that
a continuous function $f\in C(\d\Omega)$  satisfies the admissibility condition (C1) on a side $\ell$ if and only if $f$ restricted to $\ell$ is monotone.
\end{defn}

In order to present the admissibility conditions for functions which are not monotone we need more auxiliary notions. 
\begin{defn}\label{dehu}
For a given  $f\in C(\d\Omega)$,
we associate with $\ell$, a side of $\d\Omega$, 
a family of closed intervals $\{I_i\}_{i\in \cI}$ such that $
I_i =[a_i, b_i]\subsetneq \ell$ and $I_i\cap \d\ell =\emptyset$. We assume that on each $I_i$, the function  $f$ is constant and attains a local maximum or minimum and each  $I_i$ is maximal with this property. 
We will call $I_i$ a {\it hump}. In other words, maxima/minima are attained on humps.
We also set $e_i = f(I_i)$, $i\in\cI$.
\end{defn}

After this preparation, we state the admissibility condition for non-monotone functions.
\begin{defn} \label{admC2}
 A continuous function $f$, which is not monotone on a side $\ell$, satisfies the admissibility condition (C2) if and only if for each hump $I_i = [a_i, b_i]\subset\ell$, $i\in\cI$ the following inequality holds, 
 \begin{equation}\label{A}
 \dist(a_i, f^{-1}(e_i)\cap(\d\Omega\setminus I_i)) + \dist(b_i, f^{-1}(e_i)\cap(\d\Omega\setminus I_i)) < |a_i-b_i|.
\end{equation}
In addition, we require that if $y_i$, $z_i\in \d\Omega$ are such that
\begin{equation}\label{defyz}
\dist(a_i, f^{-1}(e_i)\cap(\partial\Omega\setminus I_i)) = 
\dist(a_i, y_i),\qquad 
\dist(b_i, f^{-1}(e_i)\cap(\partial\Omega\setminus I_i)) = \dist(b_i, z_i),
\end{equation}
then $y_i,$ $z_i\in \d\Omega \setminus \ell$. We will use this definition of $y_i$ and $z_i$ consistently.
\end{defn}
We note that points $y_i$ or $z_i$ need not be defined uniquely. We will  keep this in mind in our further considerations.

We use here the notation $\dist (x,\emptyset)=+\infty$. Obviously, the admissibility condition (C2) does not hold if $f$ has a strict local maximum or minimum.

\begin{remark}\label{rmk2.4} At this time we present another piece of our notation. 
If $\ell$ is a side of $\partial\Omega$, then we choose a coordinate system related to $\ell$ by requiring that $\ell$ be contained in the first coordinate axis. By our choice of the coordinate system, $\Omega$ is contained in the upper half-plane, $\Omega\subset\{ x_2>0\}$. Moreover, if $I\subset \ell$ is a hump with endpoints $a$ and $b$, which are identified with their coordinates, then $a<b$. In other words, if $\ell =[p_l,p_r]$, then by design,
$$
|p_l - a|<|p_l - b|\quad\hbox{and}\quad|p_r - b|<|p_r - a|.
$$
\end{remark}
It turns out that the examples, we present in Section \ref{examples} show that (C2) alone is not sufficient to guarantee existence of solution to (\ref{lg}), 
We need further restrictions on the data, implying that the candidates for boundaries of different level sets do not intersect.
We first present the order preserving conditions preventing intersections of level sets.
\begin{defn}\label{dfopc}
We shall say that $f\in C(\partial\Omega)$ satisfies the 
{\it order preserving condition}, (OPC for short),
if for any two different humps $I_1$, $I_2$, contained in two sides $\ell_1$, $\ell_2$, which may be equal, 
any choice of the corresponding points $y_i, z_i$, $i=1,2$ defined in (\ref{defyz}) fulfills
\begin{equation}\label{OPC}
 ([a_1,y_1] \cup [b_1, z_1]) \cap ([a_2,y_2] \cup [b_2, z_2]) = \emptyset.
\end{equation}
\end{defn}
The Order Preserving Condition rules out nonsense data but by itself it is not sufficient, see Example \ref{ex4.2}.
This is why we introduce another requirement, complementing (\ref{OPC}). In order to do so, we present a new piece of notation. 
\begin{defn}
For a given  hump   $[a,b]$ and  corresponding points $y, z$, see (\ref{defyz}) we introduce $\overline{yz}_{ab}\subset\partial\Omega$ to be the arc connecting $y$ and $z$, and
not containing the hump $[a,b]$.
\end{defn}
 For any $p\in \d\Omega$, $\epsilon>0$   we write,
$$
\cN(p,\epsilon)=B(p,\epsilon)\cap\partial\Omega\setminus\overline{yz}_{ab} .
$$
\begin{defn}\label{dfdcc}
We shall say that $f\in C(\partial\Omega)$ satisfies the 
{\it data consistency condition}, (DCC for short), if at all humps, $I=[a,b]$ contained in a side $\ell$, if there is a choice of points $y$, $z$ defined in (\ref{defyz}) such that 
\begin{equation}\label{dcc1}
\inf_{x\in \overline{yz}_{ab}} f(x) \ge f([a,b]),
\end{equation}
whenever $f$ attains a local maximum on hump $[a,b]$. Here, the points $y, z$ are defined in (\ref{defyz}) and  the arc
$\overline{yz}_{ab}\subset \p\Omega$ is defined above.
Moreover,  there is $\epsilon>0$ such that
\begin{equation}\label{dcc2}
\begin{array}{ll}
f(p_2) < f(p_1)& \hbox{if } p_1, p_2 \in \cN(z,\epsilon),\ \dist(p_2,z)<\dist(p_1,z),\\
f(p_2) < f(p_1)& \hbox{if }
p_1, p_2 \in \cN(y,\epsilon),\ \dist(p_2,y)<\dist(p_1,y)).
\end{array}
\end{equation}
Alternatively, if $f$ attains a local minimum on $[a,b]$, then we replace inf by sup and we reverse the inequalities.
\end{defn}
\begin{remark}
According to our definition, humps do not contain  endpoints of $\ell$.
However, it may happen that $f$ is constant on a subinterval of $\ell$ containing an endpoint of $\ell$, which is not a hump. We will carefully address such a situation in the course of proof of Lemma \ref{le2}.
\end{remark}

We would like to discover the consequences of the admissibility conditions.
In particular, we would like to know if the restriction of $f$ to a side 
$\ell$ can have an infinite number of local minima or maxima. Interestingly, the
answer depends upon the geometry of $\Omega$. Namely, we can prove the following
statement.

\begin{proposition}\label{pr-2.1}
Let us suppose $\ell$ is a side of the boundary of $\Omega$. In addition, $\ell$
makes an obtuse angle with the rest of $\partial\Omega$ at its endpoints.
If $f$ satisfies on $\ell$ the admissibility
condition (C2) and the OPC, then $f|_\ell$ has a finite number of humps.
\end{proposition}

\begin{proof}
Let us introduce a strip $S(\ell)$, defined as follows,
$$
S(\ell) = (\bigcup_{x\in\ell} L_x)\cap \Omega,
$$
where $L_x$ is the line perpendicular to $\ell$ and passing through $x$. We notice that the intersection of the  boundary of $S(\ell)$ with 
$\Omega$ consists of two line segments,
$$
\partial S(\ell) \cap \Omega = s_l \cup s_r.
$$
We follow  the convention specified above in Remark \ref{rmk2.4} and we denote $\ell$ by $[p_l,p_r]$.  We assume that $s_m$ passes through $p_m$, where $m= l, r$.

Given a hump $[a_i,b_i]$, we can have the following situations: each of $[a_i, y_i], [b_i, z_i] $ may either be contained in $S(\ell)$  
or 
may intersect $s_l \cup s_r$, here $y_i$'s and $z_i$'s are defined in (\ref{defyz}).

We claim that there are only a finite number of segments $[a_i, y_i], [b_i, z_i] $ 
contained in $S(\ell)$. Indeed, if it were otherwise, then 
there  would be a sequence $k_n$ such that
$|b_{k_n} - a_{k_n}|\to 0$ as $k_n$ goes to infinity. On the other hand,
\begin{equation}\label{czero}
\min\{ \dist(a_i, \partial \Omega\cap S(\ell)\setminus\ell), \dist(b_i, \partial \Omega\cap S(\ell)\setminus\ell)\}\ge c_0>0.
\end{equation}
But these two conditions combined contradict the admissibility conditions (C2).

In the next step, we claim that there are only a finite number of segments $[a_i, y_i], [b_i, z_i] $
intersecting $s_l \cup s_r =  \partial S(\ell) \cap \Omega $.  Let us  suppose otherwise, then 
we notice that in case of  an infinite  number of humps contained in a side $\ell$ condition (\ref{czero}) is no longer at our disposal.

We notice that the only accumulation point of these humps may be an endpoint of 
$\ell$, for otherwise (\ref{czero}) would be valid implying the violation of the admissibility condition (C2). Let us take a subsequence of humps $[a_{i},b_{i}]$ with the length converging to zero and 
with the endpoints  converging to an endpoint of $\partial \ell$.  
Without the
loss of generality, we can assume that $a_i\to p_l$.

We claim that for $a_i$ sufficiently close to $p_l$, the interval $[a_i, y_i]$ intersects $s_l$. Let us suppose otherwise, i.e.  $[a_i, y_i]$ intersects $s_r$ or  $[a_i, y_i]$ does not intersect neither  $s_l$ nor $s_r$. For large $i$, we have $|b_i - a_i|< \frac 18 \min\{ |p_r - p_l|, |s_r|, |s_l|\}$. If $[a_i, y_i]\cap s_r\neq \emptyset,$ then $|a_i - y_i|> \frac 12  |p_r - p_l|$. If $[a_i, y_i]$ does not intersect neither  $s_l$ nor $s_r$, then $|a_i - y_i|> \frac 12 \min\{ |s_r|, |s_l|\}$.
In both cases condition (C2) implies,
$$
\frac 18 \min\{ |p_r - p_l|, |s_r|, |s_l|\}> |a_i - b_i|>
|a_i - y_i| + |b_i - z_i| >\frac 12 \min\{ |p_r - p_l|, |s_r|, |s_l|\},
$$
yielding a contradiction.

As a result, infinitely many  $[a_i, y_i]$ must intersect $s_l$. In this case, there exists a hump $[a_k, b_k]$ such that $[a_k, b_k]\subset [p_l, a_i]$, then both $[a_k,y_k]$ and $[b_k,z_k]$ must intersect $s_l$ due to the Order Preserving Condition, (\ref{OPC}).

Since both segments $[a_k, y_k]$ and $[b_k,z_k]$ intersect $s_l$ and
the angle between $\ell$ and $\partial \Omega$ is obtuse,
then by a simple geometry condition we deduce that \eqref{A} is violated again. 
\end{proof}

Actually, we will make the above statement even  more precise. 

\begin{proposition}\label{p-2.2}
Let us consider $\ell$, a side of $\d\Omega$, 
the strip $\partial S( \ell)\cap \Omega$ and the set of all humps contained in $\ell$, $\{I_i\}_{i\in \cI}$, where $ I_i = [a_i,b_i]$. If $\partial\Omega$ forms obtuse angles at $p_m\in\partial \ell$, where $m=l$ or $m=r$, 
then 
$$
s_m \cap \left( \bigcup_{i\in \cI}([a_i,y_i]\cup [b_i,z_i])\right), 
$$ 
consists of at most one point. Here, $y_i,z_i$, $i\in \cI$ are any points satisfing (\ref{defyz}).  
\end{proposition}
\begin{proof} We may assume for the sake of definiteness that $p_m = p_l.$
We use the notation convention introduced in Remark \ref{rmk2.4}, in particular,
$$
\dist(p_l, a_i)<\dist(p_l, b_i), \qquad \dist(p_r, b_i)< \dist(p_r, a_i).
$$

Let us suppose that our claim does not hold and the set
$s_l \cap \left( \bigcup_{i\in \cI}([a_i,y_i]\cup [b_i,z_i])\right)$ contains more than one element.
(The argument for $s_r \cap \left( \bigcup_{i\in \cI}([a_i,y_i]\cup [b_i,z_i])\right)$ will be the same.)


Due to (\ref{A}), we conclude that  $[a_i,y_i] \cap [b_i,z_i] =\emptyset$. Thus, if $[b_i, z_i]$ intersects $s_l$, then  
$[a_i, z_i]$ intersects $s_l$ too. Then, the geometry implies that the admissibility condition (C2) is violated.

Let us suppose that $[a_i,y_i]$ and $[a_k,y_k]$ intersect $s_l$, $j\neq k$ and $|a_i-p_l|<|a_k-p_l|$. In this case, the OPC 
implies that $[b_i, z_i]$ intersects $s_l$. If this happens,  then we are back to the case we have just discussed, hence 
the admissibility condition (C2) is violated. Our claim follows.
\end{proof}
%
%
%

We will make further  observations about the structure of admissibility conditions. A particularly interesting one is the case when $\partial\Omega$ has an infinite number of sides. For the sake of simplicity,  we will assume that $\{\ell_k\}_{k \in \cK}$ has at most one accumulation point, i.e. if $\ell_k = [p^k_l, p^k_r]$, then $p^k_l, p^k_r \to p_0$, as $k$ goes to infinity. 
We assume that $\ell_k= [p^k_l, p^k_r]$ are  such  that $\dist( p^k_l,p_0) > \dist( p^k_r,p_0)$ with $p_0$ an endpoint of a side $\ell_0$. We note that the same argument applies in case of a finite number of accumulation points.

 
After this preparation, we will see what kind of restrictions impose the admissibility condition (C2) on the boundary data as well as  on any sequence of sides,  accumulating at $p_0$. The boundary $\partial\Omega$ at $p_0$ may have a tangent line or form an angle. The angle may be obtuse (including the case of a tangent line) or acute. We shall see that the measure of the angles plays a major role. Namely,
we show that:\\
(a) if the angle is obtuse, then $p_0$ may  not be any accumulation point of any sequence of humps contained in $\ell$. \\
(b) if the angle at $p_0$ is acute, then we may have an infinite number of humps on $\ell$,  accumulating at $p_0$.


In the next Lemma, we will consider the case (a). An example to support (b) is presented in the last section, see Example 4.5.

\begin{lemma}\label{lm2.0}
We assume that $p_0$ is an accumulation point of  $\{\ell_k\}_{k=1}^\infty$, where $p_0$ is an endpoint of a side $\ell_0$.
Moreover, 
$\partial\Omega$ forms an obtuse angle at $p_0$, 
and $f\in C(\partial\Omega)$ satisfies the admissibility condition 
(C2) on the sides of $\partial \Omega$.
Then, there is $\rho>0$ with the following properties:\\
(1) If $\ell_k \subset  B(p_0,\rho)$, then $\ell_k$ contains at most one hump $I_k$.\\
(2) If $I_k$ is the hump of
$\ell_k \subset B(p_0, \rho)$ mentioned in (a), 
then intervals $[a_k,y_k], [b_k,z_k]$, 
must intersect  $\partial S(\ell_k) \cap \Omega$.
\end{lemma}

\begin{proof}
There is $\rho>0$ such that every $\ell_k \subset B(p_0,\rho)$ forms obtuse angles with its neighbors. Due to the obtuse angle at $p_0$, the length of each component of $\partial S(\ell_k)$ may be made strictly bigger than a fixed number $c_0$, see (\ref{czero}), while the length of $\ell_k$ goes to zero. We further restrict $\rho$, by requiring that all sides $\ell_k$, contained in  $B(p_0,\rho)$, have length smaller than $c_0$. Thus, existence of 
a hump $[a,b]\subset \ell_k \subset B(p_0,\rho)$, such that $[a,y]$ or $[b,z]$ are contained in $S(\ell_k)$, violates (C2). Moreover, Lemma \ref{p-2.2} implies that at most one interval of the form $[a,y]$ (resp. $[b,z]$) intersects $\d S(\ell_k)$, where $[a,b]$ is a hump contained in $\ell_k$. In other words, any side $\ell_k \subset B(p_0,\rho)$ may have at most one hump. This  observation implies part (2) too.
\end{proof}

\section{Construction of solutions for continuous data}\label{sCdata}

Solutions to (\ref{lg}) are constructed by a similar limiting process used in \cite{grs}. To prove Theorem \ref{tsuf-C}, we first find a sequence of strictly convex domains, $\{\Omega_n\}_{n=1}^\infty$, approximating $\Omega$. Then, we define $f_n$ on $\d\Omega_n$ in a suitable way. After this preparation, we invoke the classical result in \cite{sternberg}, to conclude existence of $\{v_n\}_{n=1}^\infty$, solutions to the least gradient problem in $\Omega_n$ with data $f_n$ on the boundary of $\Omega_n$. 
This approach is good, when $\d\Omega$ has finitely many sides. The case of infinitely many sides is dealt  in a separate section.

\subsection{Case of finitely many sides of $\partial\Omega$, with finitely many humps}\label{finite Cont}
The construction of strictly convex region $\Omega_n$ is straightforward, when $\d\Omega$ is a polygon with a finite number of sides.  This is  the content of the following Lemma.

\begin{lemma}\label{prap}
 Let  us suppose that $\Omega\subset\bR^2$ is a convex region with finitely many sides $\ell_k$, $k\in\cK=\{1,\ldots, K\}$. Then,
 there is a sequence of strictly convex bounded regions, $\Omega_n$ containing $\Omega$ and such that $\bar\Omega_n$  converges to
 $\bar\Omega$ in the Hausdorff metric.  Moreover, if $n\leq m$, then
$$\Omega_m\subseteq\Omega_n.$$
\end{lemma}

\begin{proof} 
For  each vertex $p\in \partial \Omega$, we select $L(p)$ a line passing through $p$ such that $L(p)\cap \bar \Omega=\{p\}$. 

For each side $\ell_i=[p_i^1,p_i^2]$, $i=1,\ldots, K$,
of $\partial \Omega$, we consider the positively 
oriented coordinate system with origin $p_i^1=
(0,0)$, such that $\Omega$ is in the upper half 
plane, and $\ell_i=[0,d_i]\times \{0\}$. On each side, we
construct an increasing sequence  of strictly 
convex functions $\kappa^n_i:[0,d_i]\to \mathbb R$, with
$\kappa^n_i(0)=0$, $\kappa^n_i(d_i)=0$, 
$$
\dfrac{d\kappa^n_i}{dx}(0)= \frac 1n s_1,\qquad \dfrac{d\kappa^n_i}{dx}(d_i) = \frac 1n s_2,
$$ 
where $s_k$ is the slope of $L(p_i^k)$, $k=1,2$
and such that the Hausdorff distance between 
$c^n_i:=\hbox{graph}\,(\kappa^n_i)$ and $\ell_i$ is smaller
than $\frac{1}{n}$. We notice that 
$\Omega_n=\bigcup_i c_n^i$ is strictly convex and 
$\Omega_m\subseteq \Omega_n$ if $m\geq n$. Finally, the construction of
$\bar\Omega_n$ implies that $\bar\Omega_n$ converges to $\bar\Omega$ in the Hausdorff metric.
\end{proof}

We are now ready to define the trace functions $f_n$ on $\Omega_n$. 

\begin{defn}\label{def:fn}
Assume $\Omega$ is an open convex set whose boundary is a polygon with  sides $\ell_k$, $k\in \mathcal K=\{1,\ldots,K\}$. Let $\Omega_n$ be the strictly convex sets as constructed in Lemma \ref{prap}.
We take the orthogonal projection $\pi:\bR^2 \to \bar\Omega$ onto a given convex closed set, see \cite{brezis}. Then, we define
$f_n\in C(\partial\Omega_n)$ by the following formula,
\begin{equation}\label{defn}
 f_n(y):= f(\pi( y)).
\end{equation}
\end{defn}


%


This definition preserves continuity properties of $f$.

\begin{proposition}\label{c0}
(1) If $f_n$ is defined above, then $\omega_f$, the continuity modulus of $f$, is also the continuity modulus of $f_n$. \\
(2) Let us set $\pi_n = \pi|_{\p\Omega_n}$. If we take $x\in \p\Omega$ and $y_n \in \pi_n^{-1}(x)$, then $y_n \to x$ as $n\to+\infty$.
\end{proposition}
\begin{proof}
(1) The argument is based on the observation that if $x_1, x_2\in \d\Omega_n$, then 
 $|\pi x_1 - \pi x_2 | \le  |x_1 -  x_2 |$, see \cite[Proposition 5.3]{brezis}. This implies our claim.
 
(2) 
For a given $\epsilon>0$, since $\bar\Omega_n$ converges to $\bar\Omega$ in the Hausdorff metric, then we deduce that any $y_n\in \pi^{-1}_n(x)$ must be at a distance from $\Omega$ smaller  than $\epsilon$, i.e. $|y_n - x|<\epsilon$. 
\end{proof} 


Now, we assume that $f$ satisfies admissibility condition 
(C2) and $f$ has finitely many humps. We have to check that 
the distances, appearing in (\ref{A}), are well approximated by $f_n$, in the sense explained below. Let us assume  that the number of sides of $\d\Omega$ is finite, $\Omega_n$ are constructed in 
Lemma \ref{prap} and $f_n$ are defined above in (\ref{defn}). We assume that $I_i= [a_i, b_i]$ is a hump contained in  $\ell$. We denote the orthogonal projection onto the line containing $\ell$ by $\pi_\ell$. We set,
$$
\alpha^n_i := \pi_\ell^{-1}(a_i) \cap \d\Omega_n,\qquad
\beta^n_i := \pi_\ell^{-1}(b_i) \cap \d\Omega_n.
$$
By the properties of $\pi$, we have $\alpha^n_i = \pi^{-1}(a_i)\cap \d\Omega_n$, $\beta^n_i := \pi^{-1}(b_i) \cap \d\Omega_n.$

We also denote by $y_i$, $z_i$ the points of $\d\Omega \setminus \ell$ defined in  (\ref{defyz}). 
Assume
$y_i \in \ell'$ and $z_i \in \ell''$.  We consider the orthogonal projection $\pi_{\ell'}$, (resp. $\pi_{\ell''}$), onto  $\ell'$,  (resp. $\pi_{\ell''}$). We take
$$
\zeta^n_i := \pi_{\ell'}^{-1}(y_i) \cap \d\Omega_n,\qquad
\psi^n_i := \pi_{\ell''}^{-1}(z_i) \cap \d\Omega_n.
$$ 
Since, the number of humps is finite and
$$
\lim_{n\to\infty} \dist(\alpha^n_i,\zeta^n_i) + \dist(\beta^n_i,\psi^n_i )
= \dist(a_i, y_i) + \dist(b_i, z_i) <|a_i -b_i|
$$
we conclude that 
\begin{equation}\label{rn-ap}
\dist(\alpha^n_i,\zeta^n_i) + \dist(\beta^n_i,\psi^n_i )  <|a_i -b_i|
\end{equation}
for sufficiently large $n$. Thus,
we have shown:

\begin{cor}\label{c1} Let us suppose that $\Omega$ is open, bounded and convex and the number of sides of $\d\Omega$ is finite. We assume that
the function  $f\in C(\d\Omega)$ satisfies the admissibility 
condition (C2), and it has a finite number of humps. Then, for sufficiently large $n$ (\ref{rn-ap}) holds for all 
$i=1, \ldots, K$, where $\Omega_n$ is constructed in Lemma \ref{prap} and $f_n$ is given by \eqref{defn}. \qed
\end{cor}

We state here a lemma saying that $v_n$, a solution to the least gradient problem on $\Omega_n$, with data $f_n$, has the  level sets predicted by the positions on $y_i$'s (resp. $z_i$'s) corresponding to $a_i$'s (resp. $b_i$'s).

\begin{lemma}\label{key-l}
Let us assume that $f$ satisfies the admissibility conditions (C1) or (C2) as well as OPC and DCC. We assume that  $[a,b]\subset \ell$ is a hump and $\ell$ is a side of $\d\Omega$ and that points $\alpha^n, \beta^n, \zeta^n$ and $\psi^n$ are defined above. Then, for sufficiently large $n$:\\  (a)
the quadrilateral $Q_n=\conv(\alpha^n, \beta^n, \zeta^n, \psi^n)$ is contained in $E^n_e=\{ v_n \ge e\}$, where $e = f([a,b])$.\\
(b) If in addition, points $y,$ $z$ are uniquely defined by (\ref{defyz}),  then the intervals $[\alpha^n,\zeta^n]$ and $[\beta^n,\psi^n]$ are subsets of $\partial E^n_e$.
\end{lemma}
\begin{proof} 
Our reasoning  is based on the DCC.
We present the argument for the case of $f$ attaining a local maximum at $[a,b]$. The other case, when a minimum occurs, can be treated similarly. We choose $y$ and $z$ so that (\ref{dcc1}) holds.

We define $\delta$ to be
\begin{equation}\label{dede}
    \delta = |a-b| - |a-y| - |b-z|>0.
\end{equation}
We also take $\epsilon>0$, as in the DCC, and if necessary we restrict it further, so that
\begin{equation}\label{cztery}
U_y:= \mathcal{N}(y, \epsilon)\subset\ell_y,\quad U_z:= \mathcal{N}(z, \epsilon)\subset\ell_z,\quad
U_a:= B(a, \epsilon)\cap \d\Omega \subset\ell,\quad U_b:=B(b, \epsilon)\cap \d\Omega \subset\ell,
\end{equation}
where $\ell_y$ and $\ell_z$ are sides containing $y$ and $z$, respectively. Possibly, $y$ or $z$ is an endpoint of $\ell_y$ or $\ell_z$. For  this $\epsilon$ we set,
$$
\sigma:= \max\{ \min \{f(x):\ x\in U_i\} \ i =a,b,y,z\}.
$$
From \eqref{defn}, we have
$$
f_n(\alpha_n) = f(\pi \alpha_n) = f(a) = e=
f(b) = f(\pi \beta_n) = f_n(\beta_n)
$$
and similarly
$$
f_n( \zeta_n) = f(y) = e= f(z) = f_n(\psi_n).
$$

We take any $\tau\in (\sigma, e)$. We will investigate $E^n_\tau$ for all $\tau\in (\sigma,e)$. Before we do so, we make a number of observations. For an arbitrary $\tau\in(\sigma,e)$, we find
$$
x^\tau_j \in U_j,\quad j=a,b,y,z
$$
such that $f(x^\tau_j) =\tau$. The Data Consistency Condition implies the strict monotonicity of $f$ restricted to each of $U_j$, hence the choice of $x^\tau_j$, $j=a,b,y,z$, is unique. Moreover, for the same reason
\begin{equation}\label{szesc}
\lim_{\tau \to e} x^\tau_j = j,\qquad j=a,b,y,z.
\end{equation}
For  a fixed $\tau<e$  there exist
$$
\xi^n_a \in \pi_n^{-1}(U_a), \quad \xi^n_b \in \pi_n^{-1}(U_b),\quad 
\xi^n_y \in \pi_n^{-1}(U_y), \quad \xi^n_z \in \pi_n^{-1}(U_z)\subset \d\Omega_n.
$$
such that $f_n(\xi^n_i) = \tau$. In principle, $\pi$ is not injective, but it is easy to see that $\pi_n = \pi |_{\d\Omega_n}$ is. Hence,
the choice of $\xi^n_j$, $j=a,b,y,z$ is unique, when $y$ and $z$ are fixed. 

By the definition of $f_n$ we have, see (\ref{dcc1}),
\begin{equation}\label{piec}
\inf_{x\in\overline{\zeta_n,\psi_n}_{\alpha_n\beta_n}} f_n(x) =
\inf_{\xi\in\overline{yz}_{ab}
} f(\xi) 
\ge f([a,b]) =
f_n(\overline{\alpha_n\beta_n})=e >\tau,
\end{equation}
where $\overline{\alpha_n\beta_n}$ denotes $\pi_n^{-1}([a,b])$. We recall that $\overline{\zeta_n,\psi_n}_{\alpha_n\beta_n}$ is the arc connecting  $\zeta_n$ with $\psi_n$ and not containing $\overline{\alpha_n\beta_n}$.

Let us denote by 
\begin{equation}\label{dfHb}
H^b([\psi^n,\zeta^n])  \hbox{ the closed half-plane whose boundary contains} [\psi^n,\zeta^n] \hbox{ but } \alpha^n,  \beta^n\not\in H^b([\psi^n,\zeta^n]).
\end{equation}
We note that any point $x'\in \Omega_n \cap H^b([\psi^n,\zeta^n])$ must be in $E^n_e$. Indeed, if  $x'\in \Omega_n \cap H^b([\psi^n,\zeta^n])$ and $v_n(x')=s' <e$, then we may possibly choose another point $x_0$ in $H^b([\psi^n,\zeta^n])$, such that $x_0\in \d E_s$, where $s<e$. Thus, $\d E_s$ must  intersect $H^b([\psi^n,\zeta^n])\cap \d \Omega_n$, but this is impossible due to (\ref{piec}). A similar argument shows that the set $H^b([\alpha^n,\beta^n]) \cap \d \Omega_n\subset E^n_e$, where $H^b([\alpha^n,\beta^n])$ is the closed half-plane whose boundary contains $[\alpha^n,\beta^n]$ and  $\psi^n$, $\zeta^n\not\in H^b([\alpha^n,\beta^n])$. This proves part (a).

We can show part (b).
Due to \cite[Lemma 3.3]{sternberg} almost all $\tau\in (\sigma,e)$ are such that $\d E^n_\tau$ intersects all sets $\pi^{-1}_n(U_i)$, $i=a,b,y,z$. We claim that $[\xi^n_a, \xi^n_y]$ and $[\xi^n_b, \xi^n_z]$ are subsets of $\d E^n_\tau$. Firstly, we claim that if $[\xi^n_a,p]$ is a connected component of  $\d E^n_\tau$, then $p=\xi^n_y$. Indeed, this follows from (\ref{szesc}) combined with $\xi^n_a\to x^\tau_a$ as $n\to \infty$, the strict monotonicity of $f_n$ on $U^n_a$ and (\ref{defyz}). The same argument works for $[\xi^n_b, \xi^n_z]$. Moreover, 
we  argue that 
\begin{equation}\label{siedem}
|\xi^n_a - \xi^n_b| + |\xi^n_y - \xi^n_z| > |\xi^n_a - \xi^n_y| + |\xi^n_b - \xi^n_z|.
\end{equation}
Indeed,  for our choice of $\delta$, we can take $\tau$ such that
$$
|x^\tau_a - x^\tau_y | \le |a - y| + \frac \delta5,\qquad 
|x^\tau_b - x^\tau_z | \le |b - z| + \frac \delta5.
$$
Similarly, for sufficiently large $n$, we have
$$
|\xi^n_a - \xi^n_y| \le |x^\tau_a - x^\tau_y |  + \frac \delta5, \qquad
|\xi^n_b - \xi^n_z|  \le |x^\tau_b - x^\tau_z | + \frac \delta5.
$$
Combining these estimates and keeping in mind the definition of $\delta$, see (\ref{dede}), gives
$$
 |\xi^n_a - \xi^n_y| + |\xi^n_b - \xi^n_z| \le |a - y| + |b-z| + \frac 45\delta < |a-b|.
$$
Moreover,  due to the properties of $\pi_n$, (see \cite[Proposition 5.4]{brezis}), we have
$$
|a -b| \le |\alpha^n - \beta^n| \le |\xi^n_a - \xi^n_b| <  |\xi^n_a - \xi^n_b| + |\xi^n_y - \xi^n_z|,
$$
as desired.

We have just proved that $[\xi^n_a , \xi^n_y] \cup [\xi^n_b , \xi^n_z]\subset \d E^n_\tau$ for a.e. $\tau\in(\sigma, e)$. Hence, for all $x_0\in Q_n$, we have $f(x_0)\ge \tau$ for a.e. $\tau\in (\sigma, e)$. As a result $f(x_0)\ge e$. 
\end{proof}

\bigskip
Our construction of solutions will be performed in a few steps. Our standing assumption is 
that OPC and DCC always hold here. In this section we treat the case of $f$ having a finite number of humps and sets $\Omega$ with finitely many sides.  In this situation we can estimate the modulus of continuity of solutions to the approximate solutions on $\Omega$. This is done in the Lemma below. 

\begin{lemma}\label{le2} Let us suppose $\Omega$ is a bounded region and $\d\Omega$ has a  finite number of sides. We assume that   $f\in C(\partial \Omega)$ has a modulus of continuity $\omega_f$ and that it has  finitely many humps. We assume
that sets $\Omega_n$ are given by Lemma \ref{prap} and $f_n\in C(\d\Omega_n)$ is as
 in Definition \ref{def:fn}.   Then, there exists $v_n$, a unique solution to the
  least gradient on $\Omega_n$ with data $f_n$. Moreover,  each $v_n$ is  continuous with
the  modulus of continuity $\omega_{v_n}$, such that there exist $A, B>0$ independent of $n$, with the properties
\begin{equation}\label{teaz-l}
\omega_{v_n}(r) \le K\omega_f\left( r B + A {\sqrt r} \right)=:\tilde\omega(r),
\end{equation}
where $K$ is the number of sides of $\d\Omega$, 
$$
A = \frac{\sqrt{\diam \Omega}}{\sqrt{\min \sin \gamma}},
$$
the minimum is taken over all angles $\gamma$'s formed by neighboring sides,
$$
B = \dfrac{1}{ 
\min\{ \min_{\ell_1 \parallel \ell_2}\sin \alpha, 
\min_{\ell_1 \not\parallel \ell_2}\sin \beta\}},
$$
the angles $\alpha$'s are defined in (\ref{d-tana}), while angles $\beta$'s are defined in (\ref{d-sinbeta}). 
\end{lemma}
\begin{remark} We stress that $\tilde\omega$
depends on $\omega_f$, $\Omega$ and on the number of sides of $\d\Omega$, as long as it is finite, and on the geometry of the data, but it
does not depend on the number of humps.
However, the presence of  $K$ in front of $\omega_f$ makes the estimate  blow up, in case of infinite number of sides. 
Thus, the case of an infinite number of sides has 
to be treated differently.
\end{remark}

{\it Proof.\ } We recall that existence of $v_n$, solutions to (\ref{lg}) for each $\Omega_n$ and continuous $f_n$, follows from \cite[Theorems 3.6 and 3.7]{sternberg}.

In order to estimate $ \omega_{v_n}$, the modulus of continuity of $v_n$,
we consider a number of cases depending on the behavior
of flat pieces near the junction with the rest of $\d\Omega$. In
\cite{grs}, we could in advance guess the structure of the level set of
the solution. This was possible due to a simple structure of $\Omega$ and the data satisfying our admissibility condition (C1).
Here, it is much more difficult, so we use for this purpose the fact that the  level set structure of $v_n$ is known. We set $E^n_t=\{ v_n(x)\ge t\}$. We know that $\d E^n_t$ is a union of line segments. In general, we know that fat level sets may occur, so there may be points $x\in \Omega_n$, which do not belong to any $\d E^n_t$. 

We will proceed by considering all possible cases.



{\bf Case I}:\ \ $x_1, x_2\in \Omega_n$ belong to the
boundaries of the superlevel sets, i.e. there are $t_1,$ $t_2$ such
that $x_i\in \d E^n_{t_i}$, $i=1,2$. 
We have to estimate 
\begin{equation}\label{wew}
v_n(x_1) - v_n(x_2) = t_1 - t_2 = f_n(\bar x_n^{t_1}) - f_n(\bar x_n^{t_2}), 
\end{equation}
for properly chosen points $\bar x_n^{t_i} \in \d\Omega_n\cap \d E^n_{t_i}$, $i=1,2$, 
in terms of the continuity modulus of $f$. Existence of $\bar x_n^{t_i}$ is guaranteed by \cite{sternberg}.

We have to estimate the distance between the points in $\d\Omega_n \cap \d E^n_{t_1}$ to the point in the
intersection of $\d\Omega_n$ and $\d E^n_{t_2}$. 
We will consider a number of subcases. Here is the first one:

(I.a) There
exist sides of $\Omega$, $\ell_1$, $\ell_2$, which are parallel and  such that $\d E^n_{t_i}$, $i=1,2$, intersect both of them. 
We will use the following shorthands, $\d E^n_{t_i} =: e_i$, $i=1,2$, see also Fig. 1. The following argument is valid for both admissibility conditions (C1) and (C2).
\begin{figure}
{\centering
\includegraphics[width=10cm,height=6cm]{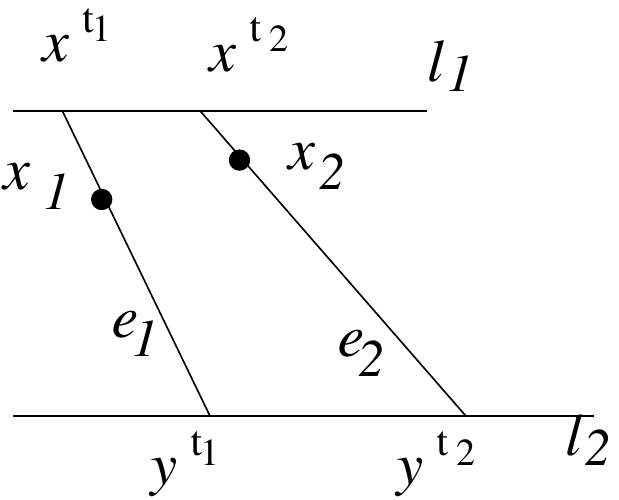}
  \caption{Case (I.a)}
  }
  \end{figure}

The first observation is obvious,
$$
|x_1 -x_2| \ge\min\{ \dist(x_1, e_2), \dist(x_2, e_1)\} \ge  \dist(e_1, e_2).
$$
Let us write $\{x^{t_1},x^{t_2}\} = \ell_1\cap( e_1 \cup e_2)$ and 
$\{y^{t_1},y^{t_2}\} = \ell_2\cap( e_1 \cup e_2)$.
If $\alpha_i$ is the acute angle formed by $e_i$ and $\ell_1$ or $\ell_2$, $i=1,2$, then 
$$
\dist(e_1, e_2) \ge \min \left\{ 
|y^{t_2} - y^{t_1}|\sin \alpha_2,|x^{t_2} - x^{t_1}|\sin \alpha_1
\right\}.
$$
We may estimate $\alpha_1$, $\alpha_2$ from below by $\alpha$, such that 
\begin{equation}\label{d-tana}
\tan \alpha = \frac{\dist(\ell_1, \ell_2)}{\diam(\pi_2 \ell_1 \cup \ell_2)} \ge
\frac{\dist(\ell_1, \ell_2)}{\diam(\Omega)},
\end{equation}
where $\pi_2$ is the orthogonal projection onto the line containing $\ell_2$.

We continue estimating the right-hand-side (RHS) of
(\ref{wew}). 
If $|x^{t_1} - x^{t_2}| <|y^{t_1} - y^{t_2}|$, then we
choose for $\bar x_n^{t_i}$ the point in $\d\Omega_n\cap \d E_{t_i}$, which is closer to $\ell_1$. Then, 
\begin{equation}\label{p-wacek}
|v_n(x_1) - v_n(x_2)| = |f_n(\bar x_n^{t_1}) - f_n(\bar x_n^{t_2})| = 
|f(\pi_1 \bar x_n^{t_1}) - f(\pi_1 \bar x_n^{t_2})|
\le 
\omega_f(| \pi_1\bar  x^{t_1} - \pi_1\bar  x^{t_2}|),
\end{equation}
where $\pi_1$ is the orthogonal projection onto the line containing $\ell_1$. We also use here the definition of $f_n$. 

We need to compare $| \pi_1\bar x_n^{t_1} - \pi_1\bar x_n^{t_2}|$ and $| x^{t_1} -  x^{t_2}|$. Let us denote by $\pi^n_1$ the orthogonal projection onto the line $L^n_1$ parallel to $\ell_1$ and passing through $\bar  x_n^{t_1}$. Then, we obviously have,
$$
| \pi_1\bar x_n^{t_1} - \pi_1\bar x_n^{t_2}| = | \bar x_n^{t_1} - \pi^n_1\bar x_n^{t_2}|
\le | \bar x_n^{t_1} - \hat x_n^{t_2}|,
$$
where $\hat x_n^{t_2}$ is the intersection of $e_2$ with line $L_1^n$. 
The last  inequality above follows from 
our construction. The same argument yields,
$$
| \bar x_n^{t_1} - \hat x_n^{t_2}| \le  |x^{t_1} - x^{t_2}|.
$$
Finally, we see,
\begin{equation}\label{wacek}
|\pi_1 \bar x_n^{t_1 } - \pi_1\bar  x_n^{t_2}|\le |x^{t_1} - x^{t_2}|\le  |x_1 - x_2|/\sin \alpha. 
\end{equation}
Hence,
$$
 |v_n(x_1) - v_n(x_2)| \le \omega_f(| x_1 - x_2|/\sin\alpha).
$$

If  $|x^{t_1} - x^{t_2}| \ge |y^{t_1} - y^{t_2}|$, we continue in a similar fashion. Namely, we 
choose the points in $\d\Omega_n\cap \d E_{t_i}$, which are closer to $\ell_2$
and we call them $\bar  y_n^{t_i}\in e_i$, $i=1,2$. Using the argument as above, we conclude that
\begin{equation}\label{wacek2}
|\pi_2 \bar y_n^{t_1 } - \pi_2\bar  y_n^{t_2}|\le |y^{t_1} - y^{t_2}|,
\end{equation}
where $\pi_2$ is the orthogonal projection onto the line containing  $\ell_2$. Estimate (\ref{wacek}) is valid for $y$'s in place of $x$'s, thus we reach,
\begin{equation}\label{om_1}
 |v_n(x_1) - v_n(x_2)| \le \omega_f(|x_1 -x_2|/\sin \alpha).
\end{equation}

(I.b) The next subcase is, when $e_1$ and $e_2$ intersect $\ell_1$ and $\ell_2$, which are not parallel and $\ell_1 \cap \ell_2 =\emptyset$, see Fig. 2. We proceed as in subcase (I.a) with slight changes. In particular, the following argument is valid for both admissibility conditions (C1) and (C2).

\begin{figure}
{\centering
\includegraphics[width=10cm,height=8cm]{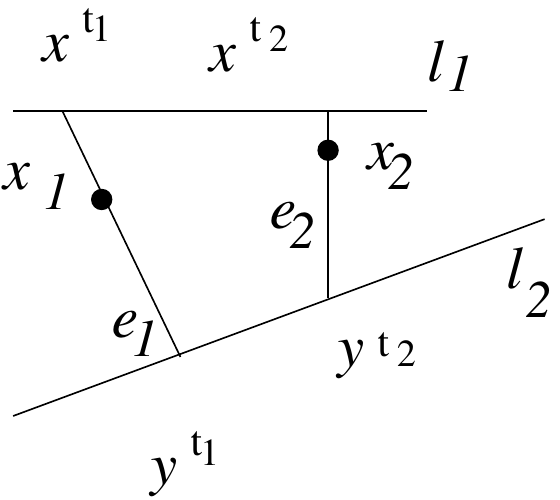}
  \caption{Case (I.b)}
  }
  \end{figure}

We have to estimate $|x_1-x_2|$ from below. In fact,
$$
|x_1 -x_2| \ge \min\{\dist(x_1, e_2), \dist(x_2, e_1)\}\ge
\min\{\dist(x^{t_1}, L(e_2)), \dist(y^{t_1}, L(e_2))\},
$$
where $L(v)$ is the line containing a (nontrivial) line segment
$v$. We notice that if $\beta_{ij}$ is the angle, which $e_i$ forms with $\ell_j$, then
$$
\sin \beta_{11}=\frac{\dist(x^{t_2},e_1)}{|x^{t_2}-x^{t_1}|}, \quad
\sin \beta_{12}=\frac{\dist(y^{t_2},e_1)}{|y^{t_2}-y^{t_1}|}, \quad
\sin \beta_{21}=\frac{\dist(x^{t_1},e_2)}{|x^{t_2}-x^{t_1}|}, \quad
\sin \beta_{22}=\frac{\dist(y^{t_1},e_2)}{|y^{t_2}-y^{t_1}|}.
$$
We want to find an estimate from below on $\beta_{ij}$. We can see that $\beta_{ij} \ge \beta$, $i,j=1,2$, where
\begin{equation}\label{d-sinbeta}
\sin\beta =\min\left\{ 
\frac{\dist(\d \ell_1, \ell_2)}{|\pi_2\ell_1|}, 
\frac{\dist(\d \ell_2, \ell_1)}{|\pi_1\ell_2|}\right\},
\end{equation}
and $\pi_i$ is the orthogonal projection onto the line $L(\ell_i)$, $i=1,2$.

Combining these estimates, we can see that
$$
|x_1-x_2| \ge \min \{ |x^{t_2}-x^{t_1}|, |y^{t_2}-y^{t_1}|\} \sin \beta.
$$
Since 
$$
v_n(x_1) - v_n(x_2) = t_1 - t_2,
$$
where $t_i = f(\bar x_n^{t_i})$ or $t_i = f(\bar y_n^{t_i})$, $i=1,2$ and $\bar x_n^{t_i}$, $\bar y_n^{t_i}\in \d \Omega_n$, 
$i=1,2$ are defined as in step (I.a), then
arguing as in subcase (I.a), we reach the same conclusion as in (\ref{wacek}) or (\ref{wacek2}).
Hence, 
\begin{equation}\label{om_2}
|v_n(x_1) - v_n(x_2)| \le \omega_f(\min\{|x^{t_2}-x^{t_1}|, |y^{t_2}-y^{t_1}|\}) \le 
\omega_f( |x_2-x_1|/\sin\beta). 
\end{equation}

\bigskip
The analysis becomes more complicated  when $e_1$ and $e_2$ intersect $\ell_1$ and $\ell_2$, which are not parallel and $\ell_1 \cap \ell_2 =\{V\}$, see Fig. 3. In these cases the admissibility conditions (C1) and (C2) come into play. The difficulty arises, when level sets may be arbitrarily close to the vertex $V$. We distinguish two situations:\\
(I.c) $f$ satisfies condition (C2) on  $\ell_1$;\\
(I.d) $f$ satisfies condition (C1) on  $\ell_1$ and on $\ell_2$.

\begin{figure}[h]\label{rys3}
{\centering
\includegraphics[width=10cm,height=6cm]{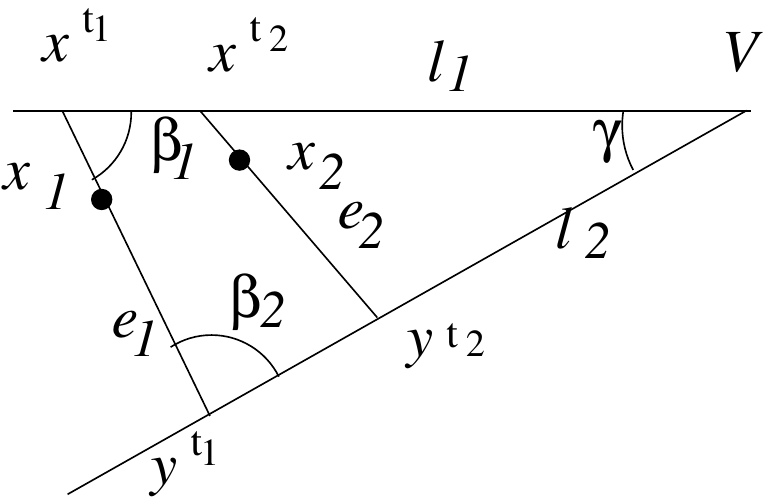}
  \caption{Cases (I.c) and (I.d)}
  }
  \end{figure}

We first consider (I.c). If this occurs, then the admissibility conditions restrict positions of $y^{t_1}$, $y^{t_2}\in \ell_2$, relative to $x^{t_1}$, $x^{t_2}$. Indeed, since we have a finite number of humps, we can find an index $i_o\in \cI$, so that
\begin{eqnarray*}
D&:=&\min\{\dist(a_{i_o},V),\dist(b_{i_o},V)\} = \dist(I_{i_o},V)\\
&=&\min \{\dist(I_j,V):\ I_j\hbox{ is a hump, }I_j\subset \ell_1\}.
\end{eqnarray*}
Of course $D\ge 0$. However, if $D=0$, i.e., the interval $[b_{i_o}, V]$ looks like a hump, then this situation is excluded by Definition \ref{dehu}.
Subsequently, we consider only $D>0$.

We denote by $z_{i_o}, w_{i_o}\in \d\Omega\setminus\ell$ such points that the distances in  (\ref{A}) are attained, i.e.,
$$
\dist(a_{i_o},z_{i_o})+\dist(b_{i_o},w_{i_o}) =
\dist(a_i, f^{-1}(e_i)\cap(\d\Omega\setminus I_i)) + \dist(b_i, f^{-1}(e_i)\cap(\d\Omega\setminus I_i)) < |a_i-b_i|.
$$
(Here, we abandon for a while our  convention of (\ref{defyz}) in order to avoid a clash of notation, because $y$'s are taken). 

We may also assume that, 
$$
\dist(b_{i_o},V) < \dist(a_{i_o},V)\quad\hbox{and} \quad
\dist(w_{i_o},V) < \dist(z_{i_o},V).
$$
We consider a triangle $T:= \bigtriangleup(V,b_{i_o},w_{i_o})$ and the following cases (i) none of points $x_1$, $x_2$ belongs to $T$, (ii) just one of $x_1$, $x_2$ belongs to $T$, (iii)  $x_1$ and $x_2$ belong to $T$.

It is obvious that (i) reduces to (I.b). Situation in (ii) may reduced to 
(iii) 
by introducing an additional point $x_3$, which is the intersection of $[x_1, x_2]$ with 
$[b_{i_o},w_{i_o}]$.  Then,
\begin{eqnarray}\label{1c-dod}
|u_n(x_1) - u_n(x_2)| &\le&
|u_n(x_1) - u_n(x_3)| + |u_n(x_3) - u_n(x_2)|\nonumber\\
&\le &
\omega_{u_n}(|x_1 - x_3| ) + \omega_{u_n}(|x_2 - x_3| ) \nonumber\\
&\le & 2\omega_{u_n}(|x_1 - x_2| ).
\end{eqnarray}

Finally, we  pay attention to (iii). In this case
points $x^{t_1}$, $x^{t_2}$ $\in \ell_1$ and  $y^{t_1}$, $y^{t_2}\in \ell_2$ are all 
in $T$. 
In this case the admissibility condition 
(C2) implies that $f$, restricted to $[b_{i_o},V]$, 
is monotone. 

We claim that $f$, restricted to $[w_{i_0},V]$, is monotone too. Let us suppose otherwise, i.e. there are $x_1,$ $x_2\in [w_{i_0},V]$ such that $f(x_1)> f(x_2)$ and $\dist{(x_1,V)} < \dist{(x_2,V)}$, we recall that, by assumption $V$ is a local minimum of $f$.
As a result there must be a local maximum of $f$ on $[w_{i_0},V]$. This maximum must be attained on a hump $[a', b']$. We call the points defined in (\ref{defyz}) by $y'$ and $z'$, respectively.

Let us suppose that $f(a') = f(b')> f(w_{i_0})$, then $y'$ and $z'$ cannot belong to $[b_{i_0},V]$. By Lemma \ref{key-l}, the quadrilateral $Q_n=\conv(\alpha^n_{i_0},\beta^n_{i_0}, w^n_{i_0},z^n_{i_0})$ is contained in $E^n_{f(b_{i_0})}$. At the same time $[a',y']$ and $[b', z']$ must intersect $Q_n$ but this is impossible, because the boundaries of the level sets cannot intersect.  Let us remark that the argument is basically the same if $f(a') = f(b')< f(V).$

Let us consider  $f(a') = f(b')< f(w_{i_0})$. If this happens, then there is an additional local minimum, which must be attained  on a hump $[a'', b''].$ 
We call the points defined by (\ref{defyz})  $y''$ and $z''$. Since $f(a'')< f(a')$ and $f$ is monotone on $[b_{i_0},V]$,  then we deduce that either $[a'',y'']$ or $[b'',z'']$ intersect $[a',y'] \cup [b',z']$, which violates the OPC or  $[a'',y'']$, $[b'',z'']$ intersect $Q_n$ defined above. However, the last event is impossible due to Lemma \ref{key-l}, as argued above. 

Hence, we conclude that $f$ restricted to $[w_{i_0},V]$ is monotone. We remark that the argument is similar if $f$ has a local maximum at $V$.

We have reached   exactly the content of the  case (I.d) considered below.

(I.d) 
In  this case, (see Fig. \ref{rys3}), due to the admissibility condition (C1) $f$ restricted to $\ell_1$ and $\ell_2$ is monotone and it attains  a minimum/maximum at $V$. Thus, we proceed as in \cite{grs}. We notice that
$$
|x_1-x_2| \ge \dist(x_2, e_1) = \min\{ \dist(x^{t_2}, e_1), \dist(y^{t_2}, e_1)\}.
$$
In addition, if $\beta_i$ is the angle formed by $e_1$ with $\ell_i$, $i=1,2$, then we notice
$$
\sin\beta_1 = \frac{\dist(x^{t_2}, e_1)}{|x^{t_2}-x^{t_1}|},\qquad
\sin\beta_2 = \frac{\dist(y^{t_2}, e_1)}{|y^{t_2}-y^{t_1}|}.
$$
While estimating $\sin\beta_i$, $i=1,2$, we have to take into account that $\ell_1$ and $\ell_2$ form 
an angle $\gamma$. Thus,
$$
\sin\gamma = \frac{d^y}{|y^{t_2}-y^{t_1}|},
$$
if $|y^{t_2}-y^{t_1}|> |x^{t_2}-x^{t_1}|$ and
$$
\sin\gamma = \frac{d^x}{|x^{t_2}-x^{t_1}|},
$$
in the opposite case. In these formulas, $d^y$ (resp. $d^x$) denotes the length of the orthogonal projection of the line segment $[y^{t_2},y^{t_1}]$  (resp. $[x^{t_2},x^{t_1}]$) on the line perpendicular to $\ell_1$  (resp. $\ell_2$). The above formulas are correct for $\gamma\in(0,\pi)$.

Thus, we can estimate $\sin \beta_i$, $i=1,2$, below as follows,
$$
\sin \beta_1 \ge \frac{d^y}{\diam(\Omega)} = \frac{\sin\gamma |y^{t_2}-y^{t_1}|}{\diam(\Omega)},
\qquad
\sin \beta_2 \ge \frac{d^x}{\diam(\Omega)} = \frac{\sin\gamma |x^{t_2}-x^{t_1}|}{\diam(\Omega)}.
$$
As a result, 
$$
|x_1-x_2| \ge \frac{\sin\gamma}{\diam\Omega}|x^{t_2}-x^{t_1}| |y^{t_2}-y^{t_1}|.
$$
Hence,
$$
\sqrt{\frac{\diam\Omega}{\sin\gamma}} \sqrt{|x_1-x_2|} \ge 
\min \{|x^{t_2}-x^{t_1}|, |y^{t_2}-y^{t_1}|\}.
$$
Arguing as in parts (I.a) and (I.b), we come to the conclusion that
\begin{equation}\label{om_3}
 |v_n(x_1) - v_n(x_2)| \le \omega_f(A \sqrt{|x_1 -x_2|}), 
\end{equation}
where $A = \sqrt{\diam\Omega}/ \sqrt{\min 
\sin\gamma}$  and the minimum here is taken over all pairs of intersecting sides.

Subcase (I.e): 
$e_1$ and $e_2$, defined earlier, intersect $\ell_1$. In addition, there are two different sides $\ell_2$ and $\ell_3$ intersecting $e_1,$ $e_2$ i.e., $e_1 \cap\ell_2 \neq \emptyset$ 
and $e_2\cap \ell_3 \neq \emptyset$. We have three possibilities corresponding to the number of points in the set,
$
\ell_1 \cap (\ell_2 \cup \ell_3).$  

We proceed  as follows. Let us suppose that $\ell_2 \cap \ell_3= \{P\}$. 
We take $t_3$ such that $\d E^n_{t_3} \cap \d\Omega$ contains $P$. If there
is no such $t_3$, then we are in the situation of {\bf Case II} considered below. 

We define   $e_3$ to be a component of $\d E^n_{t_3} $ containing $P$. Now, $e_3$
intersects segment $[x_1, x_2]$ at $x_3$ and
$\ell_1$ at $x^{t_3}$. Thus, pairs $x_1$, $x_3$ and $x_3$, $x_2$ fall into the 
known category (I.b), (I.c) or (I.d).

We have to proceed iteratively,  when $\ell_2$ and $\ell_3$ are disjoint.  Let us suppose that $\ell'_1, \ldots, \ell'_k$ is a chain of sides joining $\ell_2$ and $\ell_3$ (and different from them), i.e.,
$$
\ell_2 \cap \ell'_1 \neq \emptyset, \quad
\ell'_i \cap \ell'_{i+1}\neq \emptyset, i=1,\ldots, k-1,\quad
\ell_3 \cap \ell'_k \neq \emptyset.
$$
Now, we use the argument above for each of the pairs of sides. We
assume existence of $\tau'_i\in \bR$, $i=0,\ldots, k$ such that
$$
\ell_2 \cap \ell'_1 \in \d E_{\tau_0'}, \quad 
\ell'_i \cap \ell'_{i+1}\in \d E_{\tau_i'}, i=1,\ldots, k-1,\quad
\ell_3 \cap \ell'_k \in  \d E_{\tau_k'}.
$$
Otherwise, i.e. if one of such $\tau'_i$ is missing, then we are in the situation discussed in {\bf Case II} below. Let
$$
\{x_0' \} = \d E_{\tau_0'}\cap [x_1,x_2], \quad 
\{x_i' \} = \d E_{\tau_i'}\cap [x_1,x_2],  i=1,\ldots, k-1,\quad 
\{x_k' \} = \d E_{\tau_k'}\cap [x_1,x_2].
$$
Then, we deduce estimate (\ref{teaz-l}) as follows. The triangle inequality yields
$$
|v_n(x_1) - v_n(x_2)| \le \sum_{i=0}^{k-1}|v_n(x_i') - v_n(x_{i+1}')|  \le \sum_{i=0}^{k-1}
\omega_f(A\sqrt{|x_i' - x_{i+1}'|} +B |x_i' - x_{i+1}'| ).
$$
By the concavity of $\omega_f$  and the square root, we have,
\begin{eqnarray}\label{wyps}
\sum_{i=0}^{k-1} \frac kk
\omega_f(A \sqrt{|x_i' - x_{i+1}'|} +B |x_i' - x_{i+1}'| )&\le& 
k \omega_f(\frac A{\sqrt k}\sqrt{|x_0' - x_{k}'|} +\frac Bk |x_0' - x_{k}'| ) \nonumber \\
&=&
k \omega_f( A 
\sqrt{\frac {|x_1 - x_2|}k} +\frac Bk |x_1 - x_2|).
\end{eqnarray}
We can bound $k$ by the number $K$ of sides. 




{\bf Case II:} $x_1$ belongs to $\d E^n_t$, while there is  no real $s$ for which point $x_2$ belongs to $\d E^n_s$. We reduce it to {\bf Case I}. Since $v_n$ is continuous, thus $v_n(x_2)=\tau$ is well-defined. We take $x_3 \in \d E^n_\tau\cap [x_1,x_2]$. As a result, couples $x_1, x_3$ and $x_3, x_2$ fall into one of the investigated categories above. 

The final {\bf Case III} is when neither $x_1$ nor $x_2$ belong to any $\d E^n_t$. Let us assume that $t_1>t_2$ (in case $t_1 =t_2$ there is nothing to prove). 
We take $x_3 \in [x_1,x_2]\cap \d
E^n_{t_1}$. Clearly, the present case reduces to the previous one, because $v_n(x_1) = v_n(x_3)$ and the couple $x_2$, $x_3$ belongs to  Case II.\qed


With the help of this Lemma we establish the first of our results, which forms the content of Theorem \ref{tsuf-C}, part (a).

\begin{theorem}\label{tm-sk}
 Let us suppose that $f\in C(\d\Omega)$, where
$\Omega$ is an open, bounded and convex set, whose boundary is a polygon and $\{\ell_j\}_{j\in \cI}$ is the finite family of sides of $\d\Omega$. 
In addition we assume that the number of humps is finite. If $f$ satisfies the
admissibility conditions (C1) or  (C2) on all sides
of $\d\Omega$, as well as the complementing ordering preservation condition, (\ref{OPC}), and the data consistency condition, (\ref{dcc1}-\ref{dcc2}), then problem (\ref{lg}) has a unique solution.%
\end{theorem}
\begin{proof}
We use Lemma \ref{prap} to find a sequence 
of  strictly convex regions,  $\Omega_n$, approximating  $\Omega$. The continuity modulus of 
the boundary function $f$ is denoted by $\omega_f$. 
We notice that all $f_n$ have continuity modulus $\omega_f$. Moreover, the conclusion of 
Corollary \ref{c1} holds. 

By Lemma \ref{le2}, 
there exists a
unique solution, $v_n$ to the least gradient problem (\ref{lg}) on
$\Omega_n$ with data $f_n$. Moreover, functions $v_n$ are equicontinuous, because their modulus of continuity is bounded by $\tilde{\omega}$ given in \eqref{teaz-l}.

By the maximum principle, see \cite{sternberg}, sequence $v_n$ is
uniformly bounded and one can show 
$\int_{\Omega_n} |Dv_n|\le M<\infty.$ 
Now, we set,
$$
u_n = \chi_\Omega v_n.
$$
From \cite[Proposition 4.1]{grs} we know that $u_n$ are least gradient functions.

Since functions $u_n$ are uniformly bounded and due to Lemma \ref{le2}, they have the common
continuity modulus $\tilde\omega$, there is a subsequence (not
relabeled) uniformly converging to $u$. The uniform convergence
implies convergence of traces, i.e. $Tu_n$ goes to $Tu$. Since $Tu_n$
tends to $f$, we shall see that $Tu = f$. Indeed, if $x\in \d\Omega$ and $y_n\in\pi_n^{-1}(x)$, then
\begin{eqnarray*}
 |u_n(x) - f(x)| & \le & | v_n(x) - v_n (y_n) | + 
 | v_n (y_n)- f(x)| \\
 &=& | v_n(x) - v_n (y_n) | +
 |f_n(y_n)- f(x)|.
\end{eqnarray*}
By definition of $f_n$, we have $f_n(y_n)= f(x)$.
Due to the last part of Proposition \ref{c0}, $y_n$ goes to $x$. Since $v_n$ converges uniformly, we conclude that the right-hand-side above converges to zero, so $Tu = f$.

Moreover, the uniform convergence of $u_n$ implies the convergence of
this sequence to $u$ in $L^1$. Hence, by classical results,
\cite{miranda}, we deduce that $u$ is a least gradient function. Since
it satisfies the boundary data, we deduce that $u$ is a solution to the
least gradient problem. Moreover, the modulus of continuity of $u$ is $\tilde\omega$.

Once we proved existence, we address the problem of uniqueness of solutions. In \cite{gorny2}, the author studied the problem of uniqueness of solutions to the least gradient problem understood in the trace sense, as  we do here. The cases of non-uniqueness are classified there and related to the possibility of different partition of `fat level sets', i.e. level sets with a positive Lebesgue measure, and with the possibility of assigning different values there. In case of continuous data and solutions, we do not have any freedom to choose values of solutions on fat level sets. Thus, \cite[Theorem 1.1]{gorny2}
implies that a solution we constructed is, in fact, unique.
\end{proof}

Finally, we show that the level sets of $u$ are as we expected.

\begin{proposition}\label{quadrupole}
Let us suppose that $f$ and $\Omega$ satisfy  the hypothesis of Theorem \ref{tm-sk}. If $[a, b]\subset \ell$ is a hump, then 
the quadrupole $Q=\conv(a, b, y, z)$ is contained in $E_e=\{ u \ge e\}$, where $e = f([a,b])$. 
\end{proposition}
\begin{proof}
The claims follow from Lemma \ref{key-l} and the uniform convergence of $u_n$.
\end{proof}

\subsection{The case of an infinite number of sides and  a finite number of humps}\label{infinite Cont} 
We treat here  the case of $\Omega$ with infinitely many sides. The approach we used in the course of proof of Theorem \ref{tm-sk} cannot be used because the estimate given by
Lemma \ref{le2} depends on the number of sides. As a result, we are forced to impose an additional condition on $f$.
It could be expressed as the  admissibility condition (C1) at the accumulation point $p_0$, i.e. $p_0$ is a local minimum or maximum and there exists a neighborhood $B(p_0,\rho)$ of $p_0$ where $f$ is monotone.

We use a similar approach as in the previous theorem. We approximate the new problem by ones we can solve. In the present case, we approximate $\Omega$
by an increasing sequence of polygonal sets $\Omega_n$ having a finite number of sides.

The theorem stated below presents the content of Theorem \ref{tsuf-C}, part (b).
\begin{theorem}\label{tm-nsk}
Let us suppose that $\Omega$ is an open, bounded and convex set, whose boundary,  $\d\Omega$ is a polygon with  an infinite  number of sides. In addition,
there exists exactly one point $p_0$ being an endpoint of a side $\ell_0$, which is an accumulation point of the sides of $\d\Omega$. We assume  that $f\in C(\d\Omega)$, where
$f$ satisfies the
admissibility conditions (C1) or  (C2) on all sides of $\d\Omega$ and the Order Preserving Condition (\ref{OPC}) and the Data Consistency Condition, (\ref{dcc1}-\ref{dcc2}) hold and the number of humps is finite.
Finally, 
$f$ attains a strict local maximum or minimum at $p_0$ and
there is $\rho_)>0$, such that $f$, restricted to each component of $(B(p_0,\rho_0)\cap \d\Omega) \setminus\{p_0\}$, is strictly monotone. Then, problem (\ref{lg}) has a unique solution $u$ belonging to $BV(\Omega)\cap C(\bar\Omega)$.
\end{theorem}

\begin{proof}
We begin with a construction of a sequence of convex sets $\Omega_n$, such that $\d\Omega_n$ is a polygon with a finite number of sides. We may assume that $f$ attains a maximum at $p_0$, the argument in  the case of a minimum is similar.

For $\rho$ given in the statement of the theorem, we consider all sides of $\d\Omega$, $\{\ell_k\}_{k=1}^\infty$, contained in $B(p_0,\rho_0)$. Since we assumed that $f$ restricted to each component of $(B(p_0,\rho_0)\cap \d\Omega) \setminus\{p_0\}$, is strictly monotone, we deduce that
sides contained in $B(p_0,\rho_0)$ have no humps, i.e. $f$ satisfies (C1) on each side $\ell'$ contained in $B(p_0,\rho_0)$. This follows from the fact that $f$ has a strict maximum at $p_0$.

We set $m_1:= \max\{f(x):\ x\in \d B(p_0,\rho_0)\cap \d\Omega\}$ and $x_1\in \bar B(p_0,\rho_0)\cap \ell_0$ to be such that $f(x_1) = m_1$. We set $y_1\in (\bar B(p_0,\rho_0)\cap \d\Omega)\setminus \ell_0$ to a point to $x_1$ such that $f(x_1) = f(y_1)$. By monotonicity at $p_0$, we know that such $y_1\in \bar B(p_0,\rho_0)$ is unique. We define,
$$
\rho_1 = \min\left\{ \frac 12 \rho_0, \dist(x_1,p_0), \dist(y_1,p_0)\right\}.
$$
Subsequently, we proceed by induction. Once $x_k,$ $y_k$, $m_k$ $\rho_k$ are set, we define
$x_{k+1},$ $y_{k+1}$, $m_{k+1}$, and $\rho_{k+1}$ as follows. We introduce
$m_{k+1} = \max\{f(x):\ x\in \d B(p_0,\rho_k)\cap \d\Omega\}$ and $x_{k+1}\in \bar B(p_0,\rho_k)\cap \ell_0$ is the point such that $f(x_{k+1}) = m_{k+1}$. We set $y_{k+1}\in (\bar B(p_0,\rho_k)\cap \d\Omega)\setminus \ell_0$ to be the only point to $x_{k+1}$ such that $f(x_{k+1}) = f(y_{k+1})$.

We define,
$$
\rho_{k+1} = \min\{ \frac 12 \rho_k, \dist(x_{k+1},p_0), \dist(y_{k+1},p_0)\}.
$$
Obviously, we have $x_{k+1}, y_{k+1}\in \bar B(p_0,\rho_k)$. Since $\rho_{k+1}\le \frac 12 \rho_k\le 2^{-k} \rho_0$, we conclude that $x_k$ and $y_k$ converge to $p_0.$

For a line segment $L$  we introduce  (cf. the definition of $H^b(L,p_0)$in (\ref{dfHb})),
$$
H(L,p_0)\quad\hbox{is the closed half-plane containing }p_0,
\hbox{ whose boundary contains } L.
$$
Define $L_n=[x_n,y_n]$ and take $H(L_n,p_0)$. 
We introsuce $\Omega_n=\Omega\setminus H(L_n, p_0)$, $n \in \mathbb{N}$
and
$$
f_n(x)=
\begin{cases} f(x) & x\in \partial \Omega\cap \partial \Omega_{n},\\ f(x_{n}) & x\in L_n .
\end{cases}
$$ 
Of course, $f_n$ satisfies the (C1) or (C2) admissibility conditions and each $\Omega_n$ has a finite number of sides. Moreover, by the choice of $x_n$ and $y_n$, functions $f_n$ satisfy the Order Preserving and Data Consistency  Conditions. 

These observations imply that we may use Theorem \ref{tm-sk} to deduce existence of $u_n$, the unique solutions to the Least Gradient Problem in $\Omega_n$ with data $f_n$, $n\in \bN$.

Clearly, $\Omega_{n-1}\subseteq \Omega_n$. 
In this section, when it is necessary, we explicitly denote by the proper subscript, the domain of definition of the trace operator.

We want to show that $T_{\partial \Omega_{n-1}} u_n=T_{\partial \Omega_{n-1}} u_{n-1}=f(x_{n-1})$ on $L_{n-1}$, where $T_{\partial \Omega_k}:BV(\Omega_k)\to L^1(\partial \Omega_k)$ denotes the trace operator.

Let us suppose that our claim does not hold, i.e. there is $\bar x\in L_{n-1}$ such that $u_n(\bar x) \neq u_{n-1}(\bar x)=f(x_{n-1})$. Without the loss of generality, we may assume that $u_{n}(\bar x) >f(x_{n-1})$. Let us set $s = \max\{ u_{n}( x):\ x \in L_{n-1}\}$. Thus, there is $\tilde{x} \in L_{n-1}$ belonging to $\partial\{ u_n > s\}\cap L_n$. As a result, the intersection of $\partial\{ u_n > s\}\cap L_n$
is non-empty and the component of $\partial\{ u_n > s\}$ passing through  $L_n$ must have endpoints in $\Omega_n$ and $\Omega\setminus\Omega_n$. But this contradicts the structure of $f$ on $\d\Omega$.

We know that $u_{n+1}|_{\Omega_n}$  is a least gradient function. Since its trace on $\partial\Omega_n$ coincides with the trace of $u_n$, we deduce that we have two solutions to the least gradient problem in $\Omega_n$. However, due to the  uniqueness of solutions, implied by Theorem \ref{tm-sk}, we infer that $u_{n+1}|_{\Omega_n} = u_n.$

We have to define a candidate for a solution at least a.e. in $\Omega$. We set,
$$
\bar u_n(x)=
\begin{cases} u_n(x) & x\in \Omega_{n}\\ f(x_{n}) & x\in \Omega\setminus\Omega_n .\end{cases}
$$
Of course, at each $x\in \Omega$, this sequence is bounded and increasing. Moreover, it is constant for $k\geq N$, for some $N$ depending on $x$, hence it has a limit everywhere,
$$
u(x)=\lim_{k\to\infty}\bar u_k(x),\qquad x \in \Omega.
$$
Moreover, the convergence is in $L^1(\Omega)$.

In order to prove that $u\in BV(\Omega)$, we use the lower semicontinuity of the $BV$ norm,
$$
\varliminf_{k\to \infty} \int_\Omega |D \bar u_k| \ge \int_\Omega |D \bar u|.
$$
By the continuity of $\bar u_k$ we have 
\begin{equation}\label{x1}
\int_\Omega |D\bar u_k| = \int_{\Omega_k} |D\bar u_k| .    
\end{equation}
Moreover, we may choose $k$ so large that $\Omega\setminus\Omega_k \subset B(p_0,\rho_0)\cap\Omega$. We recall that $f$ restricted to each component of $(\d\Omega\cap B(p_0,\rho_0))\setminus\{p_0\}$ is strictly monotone. This implies that each set $\d\{\bar u_k > t\}\cap \Omega$ has one component, for  a.e. $t\in f_k(\d\Omega_k)$. Since  $\{x\in \Omega: \bar u_k(x)=t\}$ are minimal surfaces, i.e. line segments with the  length not  exceeding $\diam \Omega$. Now, we can use the coarea formula to estimate the LHS (\ref{x1}) to deduce that
\begin{align*}
\int_{\Omega} |D\bar u_k|&= 
\int_{-\infty}^{\infty} \text{Per}(\{x\in \Omega: 
\bar u_k(x)=t\}, \Omega)\,dt \\ 
&\leq \diam (\Omega) \left(\max_{\partial \Omega} f-\min_{\partial \Omega} f\right).
\end{align*}

We will show that the limit, $u$, is a least gradient function.
Let $w\in BV_0(\Omega)$, with compact support in $\Omega$, then the support of $w$ is contained in $\Omega_k$ for all  $k\geq N$, where $N$ depends upon the support of $w$. Obviously,
\begin{align*}
\int_{\Omega} |D(u+w)|=\int_{\Omega\setminus \Omega_k} |D(u+w)|+\int_{\Omega_{k}} |D(u+w)|+\int_{\partial \Omega_k \cap \Omega} |(u+w)^{+}-(u+w)^-|
\end{align*}
By the choice of  $w$, its support is contained in $\Omega_N$, then $|Dw|=0$ in $\Omega\setminus \Omega_N$, and $w^+=w^-=0$ on $L_k$ for $k\ge N$. 
Since  we have $u=u_k$ in $\Omega_k$ and $u_k$ is a least gradient function in  $\Omega_k$ so we deduce
$$\int_{\Omega_k} |D(u+w)|=\int_{\Omega_k} |D(u_k+w)|\geq \int_{\Omega_k} |D u_k |=\int_{\Omega_k} |Du|.$$
We conclude that
$$\int_{\Omega} |D(u+w)|\geq \int_{\Omega}|Du|,$$
and therefore $u$ is a least gradient function. By construction $u$ is continuous and $u|_{\partial \Omega}=f$.

Once we proved existence, 
we address the problem of uniqueness of solutions. Due to the continuity of solutions, 
the argument is exactly as in the case of Theorem \ref{tm-sk}. 
\end{proof}


\subsection{The case of infinitely many humps}\label{inf hump}
Here, we present  our result if $f$ has   infinitely many humps.
Since a convex polygon may have at most three  acute angles,
then we deduce from Proposition \ref{pr-2.1} that there are at most three sides in $\partial \Omega$ with infinitely many humps. As a result,  without the loss of generality, we may assume that our convex domain $\Omega$ has 
one side with infinitely many humps.

\begin{theorem}\label{tm-hun}
Let us suppose that $\Omega$ is a polygonal domain as in Theorem \ref{tm-nsk}
such that only one side $\ell=[p,q]$ has infinitely many humps $I_i=[a_i,b_i]$, $i\in\bN$, accumulating at $p$. We assume that the humps are denoted in such a manner that $|p-a_i|<|p-b_i|$. Moreover, the boundary datum $f\in C(\d\Omega)$ satisfies  the admissibility conditions (C1) or (C2), as well as the Order Preserving, (\ref{OPC}), and the Data Consistency conditions (\ref{dcc1}-\ref{dcc2}). Finally, we assume that $f\in BV(\d\Omega)$  in addition to continuity. Then, there exists $u\in BV(\Omega)$, a unique solution to the least gradient problem \eqref{lg} and $u\in C(\bar\Omega)$.
\end{theorem}
\begin{proof}
We stress that the admissibility condition (C2) prevents accumulation of $a_i$ and  $b_i$ in the open interval $(p,q)$. Since  $a_i$ and $b_i$ converge to  vertex  $p$ so do points  $y_i$ and $z_i$ defined  in \eqref{defyz}, independently of their choice. The numbering of points $a_i$'s and $b_i$'s is such that $|a_{i+1} - p| < |a_i -p|.$ We can do this, because the only accumulation point of $a_i$'s and $b_i$'s is $p$. For further analysis, we fix $y_i$ and $z_i$ satisfying the DCC.

We assume that the polygonal arc (and not containing $a_i$) connecting $p$ and $y_i$ is shorter than the arc  connecting $p$ and $z_i$ (and not containing $a_i$). We name the side containing $z_i$ by $\ell'$. If $f$ on $\ell'$ satisfies (C1), then we
set $L_i=[b_i, z_i]$. If $z_i$ belongs to a segment separating humps, then we also set $L_i=[b_i, z_i]$. If  $z_i$ belongs to a hump $J=[a_i',b_i']$, where  the arc connecting $p$ and $a_i'$ is shorter than the arc  connecting $p$ and $b_i'$, then we set $L_i= [b_i, b_i'].$

Subsequently,
we introduce the domain $\Omega_i=\Omega\setminus H(L_i,p)$. As a result, $\Omega_i$ is convex and all its sides have finitely many humps. Let 
$$
f_i(x)=\begin{cases} f(x) \qquad & x\in \partial\Omega\cap \d\Omega_i,\\
f(b_i)\qquad & x\in L_i.
\end{cases}
$$
The definition of $L_i$ guarantees that $f_i$ is continuous. We claim that $f_i$ also satisfies the admissibility conditions. Indeed, $f_i$ is constant, i.e. monotone on $L_i$, so the condition (C1) holds. 

Now, we will check that $f_i$ restricted to $\ell\setminus H(L_i,p)$ satisfies (C2) condition. If $J\subset \ell$ is a hump, then the points given by (\ref{defyz}) must belong to $\d\Omega\setminus H(L_i,p)$, for otherwise the OPC condition would be violated. Moreover,  $b_i$ is a hump endpoint and there is no non-trivial interval  in $\ell$ containing $b_i$ on which $f$ attains a local maximum or minimum. Thus, $f_i$ restricted to $\ell\setminus H(L_i,p)$ satisfies (C2).

Moreover,  $f_i$ on $\ell'\setminus H(L_i,p)$ satisfies (C1) or (C2). Indeed, if $f$ restricted to $\ell'$ fulfills (C1) so does its restriction $f_i$ to a subinterval. 
Suppose now, $f_i$ on $\ell'$ satisfies (C2). If  $z_i$ belongs to the segment separating humps contained in $\ell'$, then by the argument, as in the previous paragraph, we conclude that $f_i$ restricted to $\ell'\setminus H(L_i,p)$ satisfies (C2).  In this case, by the definition of  $L_i$ and OPC, interval $L_i$ may not be intersected by any other interval  of the form $[\bar a, \bar y]$, $[\bar b, \bar z]$.

Finally, we consider the case when $z_i$ belongs to a hump $J=[a_i',b_i']$. We took $L_i = [b_i, b_i']$. If this happens, we invoke DCC to see that $f_i|_{[b_i,a_{i-1}]}$ and  $f_i|_{[b_i',a_{i-1}']}$ have the same type of monotonicity. Furthermore, by OPC, no segment of the form $[\bar a, \bar y]$, $[\bar b, \bar z]$,
may  intersect $L_i$. Here, $\bar a$, $\bar b$ are hump endpoints and $\bar a, \bar b, \bar y, \bar z\in \ell \cup \ell'$. This also shows that OPC holds for $f_i$ too. Thus, $f_i$ restricted to $\ell'$ satisfies (C2) as well as the OPC.

The construction we performed preserves DCC, because we do not change the structure of the local extrema.

The reasoning above leads to a conclusion that $\Omega_i$ and $f_i$ satisfy the assumptions of Theorem \ref{tm-nsk}. We  may invoke it to deduce existence of $u_i$ a unique solution to (\ref{lg}) in $\Omega_i$.

We claim that $u_{i+1}$ restricted to $\Omega_i$ equals $u_i$. Since $u_{i+1}$ is a
least gradient function, so is its restriction to $\Omega_i$, see \cite[Proposition 4.1]{grs}. Thus,
by the uniqueness part of Theorem \ref{tm-sk} or Theorem \ref{tm-nsk} it is sufficient to check that  $u_i$ and $u_{i+1}|_{\Omega_i}$ have the same trace on $\d\Omega_i$. In fact, it is necessary to see that
$$
u_{i+1}|_{L_i} = u_i |_{L_i}\equiv f(b_i).
$$
We may apply Proposition \ref{quadrupole} in case of all our definitions of $L_i$, to deduce that $u_{i+1}|_{L_i} =  f(b_i)$. Our claim follows.

We may now define 
$$v_i(x)=\begin{cases}
u_i(x) \qquad & x \in \Omega_i,\\
f(b_i) & x\in \Omega\setminus \Omega_i.
\end{cases}
$$
Since for $k>i$, we have
$v_k|_{\Omega_i} = u_i$, then we deduce that
$$v(x) = \lim_{i\to\infty} v_i(x)
$$ 
exists for all $x\in\Omega$. Moreover, for any compact set $K\subset \bR^2$ not containing $p$, the convergence in $\Omega\cap K$ is uniform. Since  $v_i$ are  continuous, so is the limit $v$. The uniform convergence implies that $Tu_i\to f$ on $\d\Omega\cap K$ as $i\to \infty.$ 

We claim that $v$ is continuous at $p$. We take any sequence $\{x_n\}_{n=1}^\infty\subset \Omega$ converging to $p$. By the definition of $\Omega_k$ for any $x_n$, we can find $k_n$ such that
\begin{equation}\label{qsta}
x_n \in \Omega_{k_n} \setminus \Omega_{k_n-1}. 
\end{equation}
Due to Proposition \ref{quadrupole}, we know that $Q(a_k,b_k,y_k,z_k)$ is contained in
$\{u\ge f(b_k)\} \supset \{u_k\ge f(b_k)\} $. Hence, if $x_n$ satisfies (\ref{qsta}), then 
$u(x_n) = u_{k_n}(x_n) = f(y_n),$ where $y_n \in \d\Omega\setminus \d\Omega_{k_n-1}.$ Since,
$$
\lim_{n\to\infty}| \max\{ f(y): \ y \in \d\Omega\setminus \d\Omega_{k_n-1}\} - 
\min \{ f(y): \ y \in \d\Omega\setminus \d\Omega_{k_n-1}\}  | =0,
$$
we deduce that 
$$
\lim_{n\to\infty} u(x_n) = f(p).
$$

Now, we claim that $v\in BV(\Omega)$. 
We write  $T_i = H(L_i,p)\cap \Omega,$ hence $\Omega = T_i \cup \Omega_i$.  Due to the continuity of $v$ in $\Omega$, for any $i\in\bN$, we have
$$
\int_\Omega |Dv| = \int_{\Omega\setminus T_i} |Dv| + \int_{T_i} |Dv| =
\int_{\Omega\setminus T_i} |Dv_i| + \int_{T_i} |Dv| .
$$
Thus, in order to establish our claim, it suffices to see that
$$
 \int_{T_1} |Dv| <\infty .
$$
Continuity of $v$ implies 
$$
\int_{T_i} |Dv|  = \int_{T_i\setminus T_{i+1}} |Dv_{i+1}| + \int_{T_{i+1}} |Dv| .
$$
Since we have $T_1 = \bigcup_{i=1}^\infty (T_i\setminus T_{i+1})$, then
$$
 \int_{T_1} |Dv| = \sum_{i=1}^\infty \int_{T_i\setminus T_{i+1}} |Dv_{i+1}|.
$$
We will estimate $\int_{T_i\setminus T_{i+1}} |Dv|$ by the co-area formula while using  monotonicity of $f$ on $[b_{i+1}, b_i]$. First, we set $M_i:=\max \{f(b_{i+1}), f(b_i)\}$, $m_i:= \min \{f(b_{i+1}), f(b_i)\}$ and
$$
\mathcal{D}_l^i =  \{x\in T_i\setminus T_{i+1}: v_{i+1}(x)< m_i\},\qquad
\mathcal{D}_u^i = \{x\in T_i\setminus T_{i+1}: v_{i+1}(x)> M_i\},
$$
$$
\mathcal{D}_o^i = \{x\in T_i\setminus T_{i+1}:v_{i+1}(x)\in[m_i,M_i]\}.
$$
We note,
$$
\int_{T_i\setminus T_{i+1}} |Dv| = \int_{\mathcal{D}_u^i} |Dv| + \int_{\mathcal{D}_o^i} |Dv| + \int_{\mathcal{D}_l^i} |Dv|.
$$
We will use the DCC to estimate the first and the last integral on the right-hand-side. If $f$ attains maximum (resp. minimum) on $[a_i,b_i]$, then $\max_{\overline{y_iz_i}_{a_ib_i}}\ge\min_{\overline{y_iz_i}_{a_ib_i}}\ge f(b_i)$ (resp. $\min_{\overline{y_iz_i}_{a_ib_i}}\le \max_{\overline{y_iz_i}_{a_ib_i}}\le f(b_i)$. As a result, if
$$
\{ v_{i+1} = t\} \subset \mathcal{D}_u^i \qquad \hbox{(resp. }
\{ v_i = t\} \subset \mathcal{D}_l^i \hbox{),}
$$
then 
$$
\cH^1 (\{ v_i = t\}\cap( T_i \setminus\bar T_{i+1})) \le  \diam \Omega
\qquad\hbox{for a.e. }t,
$$
because for large $i$ there is just one component of $\d \{v_i>t\}$. This is so due to the monotonicity of $f$ on $[b_{i+1}, a_i]$.
This observation combined with the coarea formula yields,
$$
\int_{\mathcal{D}_o^i} |Dv| = \int_{m_i}^{M_i}\hbox{Per}(\{v \ge  t\},  T_i \setminus \bar T_{i+1})\,dt \le 
 \diam \Omega |M_i - m_i|.
$$
Moreover,
$$
\int_{\mathcal{D}_u^i} |Dv| = \int_{M_i}^\infty \hbox{Per}(\{v\ge   t\}, T_i \setminus \bar T_{i+1})\,dt \le
(\max_{\overline{y_iz_i}_{a_ib_i}} f - M_i) \diam \Omega
$$
because the $\cH^1$ measure  of any set $\d\{v \ge  t\}$ may not exceed the measure of $\diam \Omega$. 
The same reasoning yields 
$$
\int_{\mathcal{D}_l^i} |Dv| = \int^{m_i}_{-\infty} \hbox{Per}(\{v \ge  t\}, \mathcal{D}_l^i)\,dt \le
( m_i - \min_{\overline{y_iz_i}_{a_ib_i}} f) \diam \Omega. 
$$
Since $M_i - m_i = 
|f(b_{i+1}) - f(b_i)|$ and  
$\max_{\overline{y_iz_i}_{a_ib_i}} f -M_i = f(\zeta_i) - f(z_i)$, as well as
$m_i - \min_{\overline{y_iz_i}_{a_ib_i}} f = f(\xi_i) - f(z_i).$
thus,
$$
\int_{T_i\setminus T_{i+1}} |Dv| \le \diam \Omega (| f(b_{i+1}) - f(b_i)| +  |f(c_i) - f(z_i)|)
$$
where $c_i = \xi_i$ or $c_i = \zeta_i$.
Since the set
$\{t\in \bR: \ |\{v \ge  t\}|> 0\}$ has zero Lebesgue measure, we deduce that
$$
 \int_{T_1} |Dv| \le  \diam \Omega \left(\sum_{i=1}^\infty| f(b_{i+1}) - f(b_i)|  + 
 \sum_{i=1}^\infty| f(c_i) - f(z_i)|\right)
 \le 2 \diam \Omega\, TV(f)<\infty. 
$$
The proof that $v$ is a least gradient function is exactly as in the proof of Theorem \ref{tm-nsk}. Since we have already established that $v$ has the desired trace we conclude that $v$ is a solution to (\ref{lg}). The uniqueness is shown exactly in the same manner in Theorem \ref{tm-sk} or Theorem \ref{tm-nsk}.
\end{proof}

We stress that the above proof does not make any use of the number of sides of $\d\Omega$. Thus, is is valid also if their number is infinite.

\section{Examples}\label{examples}
We present a few examples showing how our theory applies. We define $\Omega = (-L,L)\times (-1,1)$, where $L>2$ and the function
$g:(-L,L)\to \bR$ as follows $g(x) = L^2 -x^2$. Furthermore, we set
$f:\d\Omega\to \bR$  to be $f(x_1, \pm 1) = g(x_1)$ for $|x_1|\le 1$ and $f(\pm L, x_2) =0$ for $|x_2|\le 1$. We take $\lambda>0$. 

Here is the first example.
We introduce $f_\lambda(x_1, x_2) = \min\{ f(x_1, x_2), g(L-\lambda)\}.$ In the examples we present, $\lambda$ is the distance of the level set $\{f = g(L-\lambda)\}$ to each corner of $\Omega$. We have two humps $[ a^-, b^-]$ and $[a^+b^+]$ , where
$$
a^- = (-L+\lambda,-1),\quad b^-=( L- \lambda, -1)\qquad
a^+ = (-L+\lambda,1),\quad b^+=( L- \lambda, 1).
$$ 
The corresponding points $y^-, z^-$ respectively, $y^+$, $z^+$) are
$y^- = a^+$, $z^-=b^+$, (respectively, $y^+ = a^-$, $z^+=b^-$). 
Moreover, 
$$
|b^--a^-| = 2(L- \lambda), \qquad |a^--y^-| = 2 = |b^- -z^-|
$$
and the condition (\ref{A}), equivalent to the (C2) condition, reads
\begin{equation}\label{exC2}
    2 < L-\lambda .
\end{equation}
Now, we state our observations.
\begin{cor}\label{co-1}
 If $\Omega$ and $f_\lambda$ are defined above, then:\\
 (a) If $\lambda\in (0, L-2)$, then the admissibility condition (C2) holds and $u_\lambda$, a solution to (\ref{lg}), is given by the following formula,
 \begin{equation}\label{er1}
u_\lambda(x_1,x_2) = f_\lambda(x_1,1).     
 \end{equation}
 (b) If $\lambda\in [L-2, L-1)$, then the admissibility condition (C2) 
 is violated, but there is a unique solution to (\ref{lg}), which is given by (\ref{er1}).\\
 (c) If $\lambda=L-1$, then $u$ given below is  a solution to (\ref{lg}),
 $$
 u(x_1,x_2) = f_{L-1}(x_1,1).
 $$
 (d) If $\lambda> L-1$, then  the admissibility condition (C2) 
 is violated and there is no solution  to (\ref{lg}).
\end{cor}
\begin{proof}
Part (a). We have already checked that (\ref{exC2}) is equivalent to
the admissibility condition (C2), hence (C2) holds. The formula for
$u_\lambda$ is easy to find after discovering solutions in
$\Omega_n$. Finally, we notice that $u_{L-1}$ is a uniform  limit of
$u_\lambda$ as $\lambda$ goes to $L-1$. We use here the fact that an $L^1$
limit of least gradient functions is of least gradient. Moreover, the uniform convergence of $u_\lambda$ implies that
the limit has the right trace. 

Part (b) follows from the construction performed in the course of proof of Theorem \ref{tm-sk}. We notice that if $\Omega_n$ are strictly convex regions, then even if (C2) is violated, then all the sets $\d\{ u_n\ge t\}$ are vertical segments. This is so because any competitor, $v$, with horizontal boundaries of the level sets has larger $\int_\Omega |Dv|$ due to $\dist(\alpha_n,y_n)< \dist(\alpha_n,\beta_n)$ and the coarea formula.

Part (c) follows from (b) after taking a limit as $\lambda$ goes to $L-1$. By the construction the sequence $u_\lambda$ converges uniformly.

Part (d) is proved by contradiction. Let us assume that a solution, $u$, actually exists. Then, $\d\{ u \ge t\}\cap \d\Omega \subset f^{-1}(t)$ for a.e. $t$, this follows from \cite[Lemma 3.3]{sternberg}. We take such  $t\in (L-1, \lambda)$  and we consider 
$$
\cH^1(\{(x, g(x)): x\in [-t,t]\}) =: l(t).
$$
We can find $t_0$ such that $l(t_0)$ is smaller than $\lambda$ for all $t>t_0$. We construct $v$, so that $\cH^1(\{ v = t\}) = l(t)$, for $t>t_0$, but $v$ has the desired trace. Thus, by the co-area formula $\int_\Omega |Du| > \int_\Omega |Dv|$. Since the level set structure of solutions is predetermined and we have found a cheaper competitor, we infer there is no solution to (\ref{lg}) in $\Omega.$
\end{proof}

The above corollary shows that, depending upon the geometry of the level sets of solutions, (C2) need not be optimal, i.e. there may be solutions if it is violated. See also Example \ref{ex4.4}.


\begin{exe}\label{ex4.2} We shall see that violation of the OPC leads to nonexistence of solutions.
Let us consider $\Omega$ as above. We set
$$
g(x) = 
\begin{cases}
x + L-1 & x\in[-L, -L+2),\\
1 & |x| \le L-2,\\
L-1-x & x\in (L-2,L].
\end{cases}
$$
and 
$$
f(x_1,x_2)= 
\begin{cases}
g(x_1) & |x_1| \le L, \ x_2 =1,\\
- g(x_1) & |x_1| \le L, \ x_2 =-1,\\
x_2 & x_1 = L,\ |x_2| \le 1,\\
-x_2 & x_1 = -L,\ |x_2| \le 1.\\
\end{cases}
$$
Then, there are two humps, $I_{-1}=[a_{-1}, b_{-1}]$
and $I_{+1}=[a_{+1}, b_{+1}]$, where
$$
a_{-1} = (2-L,-1),\ b_{-1}= (L-2,-1),\qquad 
a_{+1} = (2-L,1),\ b_{+1}= (L-2,1).
$$
We notice that if $2+2\sqrt2<L$, then $f$ satisfies (C2).
We also find that 
$$y_{-1} =(-L,1),\ z_{-1} =(L,1),\qquad
y_{+1}=(-L,-1),\ z_{+1}=(L,-1).
$$
We see that  $[a_{-1},y_{-1}]\cap [a_{+1},y_{+1}]\neq\emptyset$ and 
 $[b_{-1},z_{-1}]\cap [b_{+1},z_{+1}]\neq\emptyset$. Hence, the OPC is violated. Since the candidates for the level sets cross, there is no solution to (\ref{lg}).\qed
\end{exe}

\begin{exe}\label{ex4.3}
Now, we show that violation of DCC may lead to non-existence. We define 
$$
\Omega_1 = \hbox{int}\,(\Omega \cup \conv(C,D,V)),
$$
where 
$$
C= (-L,1),\ D(L,1),\ V =(0,1+\alpha),
$$
where $\alpha>0.$ We also set $\beta = \dist(D,V) = \sqrt{L^2 + \alpha^2},$ $S_1 =[C,V]$, $S_2 =[ D,V]$. We define the boundary data,
$$
f(x_1,x_2) = \begin{cases}
g(x_1) & |x_1|\le L, \ x_2 =-1,\\
x_2 & |x_2|\le  1, \ |x_1| = L,\\
\frac 2\beta \dist((x_1,x_2), V) -1& (x_1,x_2) \in S_1\cup S_2.
\end{cases}
$$
Obviously, this boundary function does not satisfy DCC, but (C2) holds for $2+2\sqrt2<L$.
The data has only one hump $[a,b]$, where $a = (2-L, -1)$, $b=(L-2,-1)$ and $y = C$ and $z=D$.

We claim that the problem (\ref{lg}) with this data has no solution. Let us suppose the contrary and that $u$ is a solution. In this case, $[a,y]$ and $[b,z]$ are contained in the level set $\{u \ge 1\}.$ This is so, because otherwise there would be $t<1$ such that there would be a component of $\d\{u \ge t\}$ connecting points in $B(D,\epsilon)$ for small $\epsilon.$ As a result, points $A^t$ close to $a$ and $B^t$ close to $b$ and such that $f(A^t) = f(B^t) =t$ must be connected by a component of $\d\{u \ge t\}$. However, this implies that $[A^t, B^t] \subset [a,b] \subset \d\Omega,$ but this is not possible for any functions $u$ with trace $f$.

Furthermore, if $[a,y]$ and $[b,z]$ are contained in $\{u =1\}$, then we will construct $v\in BV(\Omega_1)$ with the same trace, but $|Du|(\Omega)> |Dv|(\Omega).$ Let us fix $\epsilon>$ and take any $t\in(1-\epsilon, 1).$ We will modify $u$ in $\{u>t 
\}=:\mathcal{D}.$ We take points  $z^t_-, z^t_+\in B(z,\delta)$ for small $\delta$ and their symmetric images with respect to the $x_2$-axis $y^t_-, y^t_+\in B(z,\delta)$ and such that $f(z^t_\pm) = f(y^t_\pm) =t$.  

We take points, $A^t,$ $B^t\in \d\Omega$, such that $A^t\in B(a,\delta)$, $B^t\in B(b,\delta)$ and $f(A^t)=f(B^t) =t$.
We can find a $C^1$ curve $c\subset \Omega$ connecting $A^t$ and $B^t$, such that $\cH^1(c) < \dist(A^t, B^t) + \epsilon$. We define $\mathcal{D}_1$ to be a region bounded by $[A^t, B^t]$ and $c$.

By the general trace theory of $BV$ functions, we can find a function $h\in W^{1,1}(\mathcal{D}_1)$, such that $h= t$ on $c$ and the trace of $h$ on $[A^{t}, B^{t}]$  is $f$ and $\| \nabla h \|_{L^1} \le (1-t) |[A^{t} -B^{t}| + \epsilon$, see \cite[Lemma 5.5]{anzellotti}. 

We set $\mathcal{D}_2^+= \{ u\ge 1-\epsilon\}\cap B(z,\delta)$ for an appropriately small $\delta$ and $\mathcal{D}_2^-$ is its symmetric image with respect to the $x_2$-axis. In $\mathcal{D}_2^+$, we define $v(x) =t$ for $x\in [z^t_+, z^t_-]$ and similarly in $\mathcal{D}_2^-$. Finally, we set $v = 1-\epsilon$ on $\mathcal{D} \setminus(\mathcal{D}_1 \cup \mathcal{D}_2^+ \cup \mathcal{D}_2^-).$ 
Now, it is easy to see by using the co-area formula that $|Dv| (\Omega) < |Du|(\Omega).$\qed
\end{exe}

We show that for certain regions the (C2) is optimal, i.e. its violation leads to non-existence.

\begin{exe}\label{ex4.4}
We set $\Omega_2 = \conv(A,B,V)$, where $A= (-L,0),$ $B=(L,0)$, $V= (0,\gamma).$ We set
$$
g(x)=
\begin{cases}
\frac{x+L}{L-\alpha} & x\in [-L,-\alpha],\\
1 & |x| \le \alpha,\\
 \frac{L-x}{L-\alpha} & x\in [\alpha,L].
\end{cases}
$$
We take any $h$  monotone decreasing function, such that $h(\dist(V, p)) = 1$, where $p\in \d\Omega_2\setminus[A,B]$ and $p_1 =\pm\alpha$. Moreover, $h(\dist(A,V)) =0.$ The only hump is $[a,b]$, where $a=(-\alpha,0)$, $b=(\alpha,0)$.
We assume that here (C2) is violated i.e.
$$
2\dist(a,p) > \dist(a,b).
$$
Let 
$$
f(x_1,x_2)=
\begin{cases}
g(x_1) & (x_1,x_2) \in [A,B],\\
h(\dist((x_1,x_2), V)) & (x_1,x_2) \in [A,V]\cup [B,V].
\end{cases}
$$
We argue that if there is a solution $u$ with the trace $f$, then $\conv(a,b,p, Sp)$ must be contained in $\{u\ge 1\}$, here $S$ is the symmetry with respect to the $x_2$-axis. We argue as in Example \ref{ex4.3}. We modify $u$ on $\{u\ge 1-\epsilon\}$ for sufficiently small $\epsilon.$ We can find an arc $\mathcal{C}\subset \Omega_2$ connecting $(-\alpha-\epsilon,0)$ with $(\alpha+\epsilon, 0)$ and such that $\mathcal{H}^1(\mathcal{C}) \le 2\alpha + 4\epsilon = \dist(a,b) + 2\epsilon.$ We can connect $(\alpha+\epsilon,0)$ to a point $z^\epsilon\in [B,V]$ (and symmetrically $(-\alpha-\epsilon,0)$ to a point $y^\epsilon\in [A,V]$) in a such a way that
$$
\dist(a,b) < \dist(a,y) + \dist(b,z)
$$
implying that 
$$
\mathcal{H}^1(\mathcal{C}) \le  \dist(a^\epsilon,y^\epsilon) + \dist(b^\epsilon,z^\epsilon.)
$$
Hence, the competitor $v$ has the same trace but $|Du|(\Omega_2) > |Dv| (\Omega_2).$ \qed
\end{exe}

Finally, we construct a region $\Omega$ and a continuous function on its boundary with infinitely many humps. 

\begin{exe}
Let $L_1>0$ be given, we take any $\alpha\in (0, \frac\pi2)$ and we take any $R> L_1/\sin(\alpha/2)$.
We define 
$\ell_1 = [0, R]\times \{0\}$, $\ell_2$ to be  a line segment of length $R$ forming an angle $\alpha$ at the origin. Moreover, 
$$
\Omega = \hbox{int}\,\hbox{conv}\,(\ell_1, \ell_2).
$$
We will call by $\ell_3$ the third side of triangle $\Omega$.

We define the sequence $L_k$ as follows 
$$
L_{2k+1} = L_1 \prod_{i=1}^k \left( 
\frac{(1-\sin \alpha)^2}{1 + \sin \alpha} - \varepsilon_i \right),
\qquad L_{2k} = L_{2k-1} \frac{ 1-\sin \alpha}{1 + \sin \alpha},\quad k\ge 1,
$$
where $0<\varepsilon_i$ is decreasing to zero and $\varepsilon_1 < \frac 12 \frac{(1-\sin \alpha)^2}{1 + \sin \alpha}$.

We  denote $a_k := L_{2k}$, $b_k = L_{2k-1}$, and define $f$ on $\ell_1$ by setting $f(x) = \frac{(-1)}{k}^{k+1}$ for $x\in (a_k, b_k)$, $k\ge1$. We extend $f$ to  $\ell_1\setminus \bigcup_{k=1}^\infty (a_k, b_k)$ by linear functions.

Let us denote by $\pi$ the orthogonal projection onto the line containing  $\ell_2$ and
$$
a_k' := \pi (a_k,0), \qquad b_k' := \pi ( b_k,0).
$$
We set  $f(x) = \frac{(-1)}{k}^{k+1}$ for $x\in (a_k', b_k')$, $k\ge1$ and define $f$ on $\ell_3$ to be equal to 1. 
We extend $f$ to  $\ell_2\setminus \bigcup_{k=1}^\infty (a_k, b_k)$ to be a continuous piecewise  linear function.

It is easy to  see that we have just proved the following fact:
Let us suppose that $\Omega$ is given above. Then, function $f$ constructed above is continuous on $\partial \Omega$ and it satisfies the admissibility condition (C2). Moreover, the OPC and the DCC hold. As a result, we constructed an instance of data satisfying the assumptions of Theorem \ref{tsuf-C}, part c. Hence, a unique solution exists in $\Omega$ with data $f$ due to Theorem \ref{tm-hun}.
\qed
\end{exe} 
 

\section*{Acknowledgement} The work of the authors was in part
supported by the Research Grant 2015/19/P/ST1/02618 financed by the
National Science Centre, Poland, entitled: Variational Problems in
Optical Engineering and Free Material Design.

PR was in part supported by  the Research Grant no 2013/11/B/ST8/04436 financed by the National
Science Centre, Poland, entitled: Topology optimization of engineering structures. An approach synthesizing the methods of:
free material design, composite design and Michell-like trusses.

The work of AS was in part performed at the University of Warsaw.

The authors also thank Professor Tomasz Lewi\'nski of Warsaw Technological University for stimulating discussions and constant encouragement.


\vspace{-.2cm}
\hspace{3cm}
\begin{wrapfigure}{l}{0.15\textwidth}
\includegraphics[width=2.5 cm, height= 1.5cm]{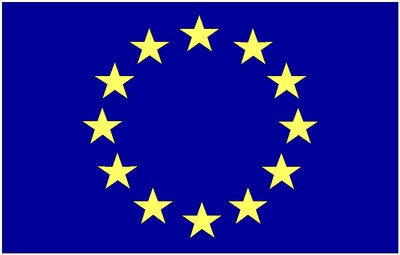}  
\end{wrapfigure}
{
This project has received funding from the European Union's Horizon 2020 research and innovation program under the Marie Sk\l{}odowska-Curie grant agreement No 665778.}
\newline
 \hspace{-5cm}

\end{document}